\documentclass[11pt]{amsart}

\usepackage{amsmath,amsthm,amssymb}
\usepackage[margin=1in]{geometry}
\usepackage[bookmarks]{hyperref}
\usepackage[all,cmtip]{xy}
\usepackage{verbatim}
\usepackage{mathtools}
\DeclarePairedDelimiter{\ceil}{\lceil}{\rceil}

\newtheorem{theorem}{Theorem}[subsection]
\newtheorem{proposition}[theorem]{Proposition}

\newtheorem{corollary}[theorem]{Corollary}
\newtheorem{lemma}[theorem]{Lemma}

\theoremstyle{definition}

\newtheorem{definition}[theorem]{Definition}
\newtheorem{example}[theorem]{Example}

\newtheorem{remark}[theorem]{Remark}

\numberwithin{equation}{subsection}

\newcommand{\gfp}{G(\mathbb{F}_p)}
\newcommand{\gfpr}{G(\mathbb{F}_{q})}

\newcommand{\res}{\operatorname{res}}

\newcommand{\ind}{\operatorname{ind}}
\newcommand{\Ext}{\operatorname{Ext}}
\newcommand{\into}{\hookrightarrow}

\newcommand{\opH}{\operatorname{H}}

\newcommand{\Hom}{\operatorname{Hom}}

\newcommand{\ga}{\gamma}

\newcommand{\la}{\lambda}

\newcommand{\al}{\alpha}
\newcommand{\ta}{\tilde{\alpha}}

\newcommand{\si}{\sigma}

\begin{document}

\title[Extensions for Finite Chevalley Groups III: Rational and Generic Cohomology]
{\bf Extensions for Finite Chevalley Groups III: Rational and Generic Cohomology}

\begin{abstract} Let $G$ be a connected reductive algebraic group and $B$ be a Borel subgroup defined over an algebraically closed 
field of characteristic $p>0$. In this paper, the authors study the existence of generic $G$-cohomology and its stability with rational $G$-cohomology groups via the use of 
methods from the authors' earlier work. New results on the vanishing of $G$ and $B$-cohomology groups are presented. Furthermore, 
vanishing ranges for the associated finite group cohomology of $G({\mathbb F}_{q})$ are established which generalize earlier work of Hiller, in addition to 
stability ranges for generic cohomology which improve on seminal work of Cline, Parshall, Scott and van der Kallen. 
\end{abstract}

\author{\sc Christopher P. Bendel}
\address
{Department of Mathematics, Statistics and Computer Science\\
University of
Wisconsin-Stout \\
Menomonie\\ WI~54751, USA}
\email{bendelc@uwstout.edu}

\author{\sc Daniel K. Nakano}
\address
{Department of Mathematics\\ University of Georgia \\
Athens\\ GA~30602, USA}
\thanks{Research of the second author was supported in part by
NSF grant DMS-1002135}
\email{nakano@math.uga.edu}

\author{\sc Cornelius Pillen}
\address{Department of Mathematics and Statistics \\ University
of South
Alabama\\
Mobile\\ AL~36688, USA}
\thanks{Research of the third author was supported in part by a Simons Foundation Collaboration Grant}
\email{pillen@southalabama.edu}

\date\today
\maketitle
\section{Introduction}

\subsection{}\label{SS:BNPintro} Let $G$ be a connected reductive algebraic group (scheme) defined over the  finite field  ${\mathbb F}_p$, 
and let $k$ be an algebraically closed field of characteristic $p$. There are two module categories for finite subgroup schemes which have close relationships to the representations of $G$ over $k$. 
One category is the modules over $G_{r}$ (the $r$th Frobenius kernel of $G$). In the case when $r=1$, this category is equivalent to 
restricted representations for the restricted Lie algebra ${\mathfrak g}=\text{Lie }G$. The other category is 
the category of representations for $G({\mathbb F}_{q})$ where $G({\mathbb F}_{q})$ is the finite group of Lie type obtained by taking 
the ${\mathbb F}_{q}$ rational points of $G$ (${\mathbb F}_{q}$ is the field of $p^{r}$-elements). Note that, for a given rational $G$-module $M$, one can consider 
the restriction of $M$ to either $G_{r}$ or to the finite group $G({\mathbb F}_{q}) $. Although there is 
no direct functorial connection between the categories of $G_{r}$-modules and $G({\mathbb F}_q) $-modules, there are 
strong interrelationships between the representation theories of $\text{Mod}(G)$, $\text{Mod}(G_{r})$ and $\text{Mod}(G({\mathbb F}_{q}))$ (cf. \cite{Hu}). 

In prior work, the authors (BNP) (cf. {\sf [BNP1-9]}) established fundamental connections between 
the cohomology theories of reductive algebraic groups, finite Chevalley groups and Frobenius kernels. The main 
idea entails the use of the induction functor $\text{ind}_{G({\mathbb F}_{q})}^{G}-$ and truncated variations of it. If $M$ is a rational $G$-module, then 
there is a natural short exact sequence: $0\rightarrow M \rightarrow \text{ind}_{G({\mathbb F}_{q})}^{G} M \rightarrow N \rightarrow 0$. This induces 
a long exact sequence in cohomology \cite[(2.1.3)] {UGA2}: 
\begin{equation} \label{eq:lesrestriction}
\begin{array}{cclclclcc}
0 & \longrightarrow & \operatorname{Hom}_G(k,M) & \stackrel{\res}{\longrightarrow} & \operatorname{Hom}_{G({\mathbb F}_{q})}(k,M) &\longrightarrow & \operatorname{Hom}_G(k,N) &&\\
& \longrightarrow & \operatorname{H}^1(G,M) & \stackrel{\res}{\longrightarrow} & \operatorname{H}^1(G({\mathbb F}_{q}),M) & \longrightarrow & \operatorname{H}^{1}(G,N) &&\\
& \longrightarrow & \operatorname{H}^2(G,M) & \stackrel{\res}{\longrightarrow} & \operatorname{H}^2(G({\mathbb F}_{q}),M) & \longrightarrow & \operatorname{H}^{2}(G,N) &&\\
& \longrightarrow & \operatorname{H}^3(G,M) & \stackrel{\res}{\longrightarrow} & \operatorname{H}^3(G({\mathbb F}_{q}),M) & \longrightarrow & \operatorname{H}^{3}(G,N) &\longrightarrow &\cdots.  
\end{array}
\end{equation}

The main advantage to this approach is that it neatly organizes information about the restriction maps in cohomology $\text{res}:\operatorname{H}^{n}(G,M)\rightarrow 
\operatorname{H}^{n}(G({\mathbb F}_{q}),M)$ and the obstructions to studying this map for each degree $n$. The first column gives information about the $G$-cohomology and 
the second about the $G({\mathbb F}_{q})$-cohomology. In the third column, the module $N$ has a $G$-filtration 
with factors of the form $M\otimes H^0(\lambda) \otimes H^{0}(\lambda^{\star})^{(r)}$ (cf. \cite[Prop. 2.4.1]{BNP8})\footnote{In the case of twisted groups, this fact has been established in 
\cite[Proposition 3.1.2]{BNPPSS} by an argument involving the Mackey Decomposition Theorem.} Here, $H^0(\la)$ (which is dual to a Weyl module) is the standard module induced from a one dimensional $B$-module (for a Borel subgroup $B$ of $G$) and $\la^{\star}$ is the weight dual to $\lambda$ (cf. Section \ref{SS:notation}). The $G$-cohomology 
$\operatorname{H}^{\bullet}(G,N)$ can then be analyzed by using the Lyndon-Hochschild-Serre spectral sequence involving $G_{r}$-cohomology and the 
geometry of the flag variety $G/B$. For example, in \cite[Thm. 3.3.1]{BNP8} we showed that an upper bound on the dimension of $\operatorname{H}^n(G({\mathbb F}_{p}),k)$ 
involves the combinatorics of the well-studied Kostant Partition Function, thus reducing the question of the vanishing of the finite group 
cohomology to a question involving the combinatorics of the underlying root system $\Phi$.


\subsection{}  In 1977, Cline, Parshall, Scott and van der Kallen \cite{CPSvdK} proved a groundbreaking result about the behavior of $G$-cohomology in relation 
to the finite group cohomology for $G({\mathbb F}_{q})$. If $V$ is a finite-dimensional rational $G$-module, then for a fixed $n\geq 0$ and $s$ and $r$ sufficiently large the restriction 
map
$$\text{res}:\operatorname{H}^{n}(G,V^{(s)})\rightarrow \operatorname{H}^{n}(G({\mathbb F}_{q}),V^{(s)})$$ 
is an isomorphism. Note that the finite group does not see the twist so 
$$\operatorname{H}^{n}(G({\mathbb F}_{q}),V^{(s)})\cong \operatorname{H}^{n}(G({\mathbb F}_{q}),V).$$
One important consequence of this theorem is, as $r$ increases, the cohomology groups $\text{H}^{n}(G({\mathbb F}_{q}),V)$ obtain a stable or generic value 
$\text{H}^{n}_{\text{gen}}(G,V)$ (also known as the generic cohomology). It is also shown in \cite{CPSvdK} what conditions for $s$ guarantee that 
$\text{H}^{n}(G,V^{(s)})\cong \text{H}^{n}_{\text{gen}}(G,V^{(s)})$ for $n=1, 2$. In recent work, Parshall, Scott and Stewart \cite[Theorem 5.8]{PSS} have shown that one can realize generic cohomology for simple modules 
with $q$-restricted dominant weights by considering the rational cohomology of simple modules with so-called shifted weights. 

The results in \cite{CPSvdK} have been used to make computations for $G({\mathbb F}_{q})$-cohomology. However, the BNP-machinery as described in Section~\ref{SS:BNPintro} lends itself 
better for this purpose because no twisting by the Frobenius is required. The BNP-techniques have been used for a variety of applications that have answered 
open questions and significantly improved known results. These include proving that self extensions vanish \cite{BNP2, BNP3}, locating the first non-trivial cohomology classes \cite{BNP8,BNP9}, computing low dimensional cohomology groups for simple modules \cite{UGA1,UGA2}, and bounding $\text{Ext}$-groups between simple modules for finite groups of Lie type \cite{BNPPSS}. Other applications appear in \cite{WW}. When these ideas were first introduced ten years ago, it was suspected that there should be a connection to the CPSvdK rational and generic cohomology results.  In this paper, we will further demonstrate that the BNP-program is widely applicable and encompassing by proving 
that the CPSvdK rational and generic stability can be obtained by using our approach. 

The paper is organized as follows. In Section 2, we establish important conditions on the existence of weights in the $B_{r}$-cohomology for certain one-dimensional representations. This enables us to 
prove powerful vanishing results on $G_{r}$, $B$ and $G$ extensions in Sections 3 and 4. For example, one key result in this paper is on the vanishing behavior of 
the $B$-cohomology groups $\text{H}^{n}(B,\mu-p^{s}\lambda)$ for weights $\lambda, \mu$ (cf. Proposition \ref{B:vanishing}). The determination of the $B$-cohomology groups with coefficients in a one-dimensional representation is still an unsolved problem and one of the first steps to understanding the structure of the line bundle cohomology of the flag variety $G/B$. 

In Section 5, the new vanishing results of Sections 3 and 4 are used to show that the cohomology groups $\operatorname{H}^{n}(G,V^{(s)})$ attain a 
stable value for $s$ large. Our proof appears to be the first proof that does not use finite group cohomology and the Main Theorem of \cite{CPSvdK}\footnote{The authors confirmed this claim
with Henning Andersen, Brian Parshall and Leonard Scott}. We then show that there is a striking connection between rational stability and the existence of a good filtration on the 
cohomology of the first Frobenius kernels $\operatorname{H}^{\bullet}(G_{1},k)$. Assuming such a filtration, one can show that $\operatorname{H}^{n}(G,V^{(s)})$ stabilizes and give robust ranges for stability (cf. Theorem 5.2.1). At the end of the section, we demonstrate how these results can be used to extend the work in \cite[Section 7]{CPSvdK}. 

The CPSvdK approach to rational and generic cohomology relies on restricting to the $B({\mathbb F}_{q})$-cohomology.  Our approach is different in that we analyze the long exact sequence  (\ref{eq:lesrestriction}) and show vanishing of terms in the right hand column.  This is done in Section 6. See in particular Proposition \ref{vanishing}.  An immediate consequence of Proposition \ref{vanishing}, presented in Section 7, is that one can establish general vanishing ranges for the cohomology group $\operatorname{H}^{\bullet}(G({\mathbb F}_{q}),k)$. In several cases, the ranges are shown to be sharp. These results improve on earlier work of Hiller \cite{H}.  

In Section 8, again using Proposition \ref{vanishing}, we determine bounds on $r, s$ as above so that $\opH^n(G,V^{(s)}) \cong \opH^n(\gfpr,V^{(s)})$ (cf. Theorems \ref{theorem:generic3}, 
\ref{theorem:generic4}, and \ref{theorem:generic5}).  Prior known bounds on $r$ and $s$ (cf. \cite{CPSvdK}, \cite{FP}, \cite{F}) depend highly on the given root system. One very important feature about our work 
is that we provide uniform bounds (which only depend on $V$, $n$ and $p$) that do not depend on the root system.  Later we demonstrate that in many cases our uniform bounds are significantly stronger than  those provided by the Main Theorem of \cite{CPSvdK}.


\subsection{Acknowledgements} The authors would like to acknowledge the American Institute of Mathematics for its support during the 2012 Workshop on Cohomology Bounds and Growth Rates. The first author thanks the University of South Alabama for its hospitality during work on this paper. 
We would also like to thank Wilberd van der Kallen and David Stewart for comments and suggestions that were incorporated into the present version of this paper. 


\subsection{Notation}\label{SS:notation}  Throughout this paper, the following basic notation will be used. 
\begin{enumerate}

\item[(1)] $k$: an algebraically closed field of characteristic $p>0$.

\item[(2)] $G$: a simple, simply connected algebraic group which is defined 
and split over the finite prime field ${\mathbb F}_p$ of characteristic $p$. The assumption
that $G$ is simple (equivalently, its root system $\Phi$ is irreducible) is largely one of convenience.
All the results of this paper extend easily to the semisimple, simply connected case.

\item[(3)] $F:G\rightarrow G$: the Frobenius morphism. 

\item[(4)] $G_r=\text{ker }F^{r}$: the $r$th Frobenius kernel of $G$. 

\item[(5)] $G^{(r)}$: the $r$th Frobenius twist of $G$; $G^{(r)} \cong G/G_r$.

\item[(6)] $G({\mathbb F}_{q})$: the associated finite Chevalley group where ${\mathbb F}_q$ is the field with $q = p^r$ elements. 

\item[(7)] $T$: a maximal split torus in $G$. 

\item[(8)] $\Phi$: the corresponding (irreducible) root system associated to $(G,T)$. When referring to short and long roots, when a root system has roots of only one length, 
all roots shall be considered as both short and long.

\item[(9)] $\Pi=\{\alpha_1,\dots,\alpha_n\}$: the set of simple roots (Bourbaki ordering).  

\item[(10)] $\Phi^{\pm}$: the positive (respectively, negative) roots.  

\item[(11)] $\alpha_0$: the maximal short root; $\ta$: the maximal root.

\item[(12)]  $B$: a Borel subgroup containing $T$ corresponding to the negative roots. 

\item[(13)] $U$: the unipotent radical of $B$.

\item[(14)] $\mathbb E$: the Euclidean space spanned by $\Phi$ with inner product $\langle\,,\,\rangle$ normalized so that $\langle\alpha,\alpha\rangle=2$ for $\alpha \in \Phi$ any short root.

\item[(15)] $\alpha^\vee=2\alpha/\langle\alpha,\alpha\rangle$: the coroot of $\alpha\in \Phi$.

\item[(16)] $\rho$: the Weyl weight defined by $\rho=\frac{1}{2}\sum_{\alpha\in\Phi^+}\alpha$.

\item[(17)] $h$: the Coxeter number of $\Phi$, given by $h=\langle\rho,\alpha_0^{\vee} \rangle+1$.

\item[(18)] $W=\langle s_{\alpha_1},\dots,s_{\alpha_n}\rangle\subset{\mathbb O}({\mathbb E})$: the Weyl group of $\Phi$, generated by the orthogonal reflections $s_{\alpha_i}$, $1\leq i\leq n$. For $\alpha\in \Phi$, $s_\alpha:{\mathbb E} \to{\mathbb E}$ is the orthogonal reflection in the hyperplane $H_\alpha\subset \mathbb E$ of vectors orthogonal to $\alpha$.

\item[(19)] $X(T)=\mathbb Z \omega_1\oplus\cdots\oplus{\mathbb Z}\omega_n$: the weight lattice, where the fundamental dominant weights $\omega_i\in{\mathbb E}$ are defined by $\langle\omega_i,\alpha_j^\vee\rangle=\delta_{ij}$, $1\leq i,j\leq n$.

\item[(20)] $X(T)_{+}={\mathbb N}\omega_1+\cdots+{\mathbb N}\omega_n$: the cone of dominant weights.

\item[(21)] $X_{r}(T)=\{\lambda\in X(T)_+: 0\leq \langle\lambda,\alpha^\vee\rangle<p^{r},\,\,\forall \alpha\in\Pi\}$: the set of $p^{r}$-restricted dominant weights. 

\item[(22)] $\leq$ on $X(T)$: a partial ordering of weights, for $\lambda, \mu \in X(T)$, $\mu\leq \lambda$ if and only if $\lambda-\mu$ is a linear combination 
of simple roots with non-negative integral coefficients. 

\item[(23)] $\lambda^\star : = -w_0\lambda$: where $w_0$ is the longest word in the Weyl group $W$ and $\lambda\in X(T)$. 

\item[(24)] $M^{(s)}$:  the module obtained by composing the underlying representation for 
a rational $G$-module $M$ with $F^{s}$.

\item[(25)] $H^0(\lambda) := \operatorname{ind}_B^G\lambda$, $\lambda\in X(T)_{+}$: the induced module whose character is provided by Weyl's character formula.  

\item[(26)] $V(\lambda)$, $\lambda\in X(T)_{+}$: the Weyl module of highest weight $\lambda$. Thus, $V(\lambda)\cong H^0(\lambda^\star)^*$.

\item[(27)] $L(\lambda)$: the irreducible finite dimensional $G$-module with highest weight $\lambda\in X(T)_{+}$. 

\item[(28)] Type $A_1$ weight notation: for $\Phi$ of type $A_1$, let $\omega := \omega_1$ denote the unique fundamental dominant weight, and $\al$ denote the unique simple root.  For a weight $\si \in X(T)$, $\si = \langle \si,\al^{\vee}\rangle\omega$. As such, we will equate $\si$ with $\langle \si,\al^{\vee}\rangle$ and simply denote weights in $X(T)$ by integers.

\end{enumerate}

\section{$B_s$-cohomology}


\subsection{}\label{SS:spectral} In this section we will investigate cohomology for the $s$th Frobenius kernel of $B$. We use here $s$ instead of $r$ to avoid confusion when 
we will apply these results when twisting a rational module $s$ times. Let $M$ be a rational $B_{s}$-module and $\mathfrak{u}=\text{Lie }U$. 
We make use of the following spectral sequence \cite[I.9.14]{Jan}. For odd $p$ define
\begin{equation}\label{Ur:spectral-odd}
E_1^{i,j}=\bigoplus M\otimes S^{a_0}(\mathfrak{u}^*)
\otimes S^{a_1}(\mathfrak{u}^*)^{(1)}
\otimes \; \cdots \; \otimes S^{a_s}(\mathfrak{u}^*)^{(s)}\otimes \Lambda^{b_0}(\mathfrak{u}^*) \otimes \Lambda^{b_1}(\mathfrak{u}^*)^{(1)}
\; \cdots\; \otimes \Lambda^{b_{s}}(\mathfrak{u}^*)^{(s)}.
\end{equation}
The summation runs over all $(s+1)$-tuples $\{a_0, \dots , a_s\}$ and $\{b_0, \dots ,b_{s}\}$ of non-negative integers satisfying $a_0=b_s=0$, 
\begin{equation}\label{U:indices-odd}
i +j = \sum_{n= 0}^{s}(2a_n +b_n) \mbox{   and     } i = \sum_{n= 1}^{s} \left(a_n p^n + b_{n-1} p^{n-1}\right).
\end{equation}
If $p=2$ set
\begin{equation}\label{Ur:spectral-even}
E_1^{i,j}=\bigoplus M\otimes S^{a_1}(\mathfrak{u}^*)
\otimes S^{a_2}(\mathfrak{u}^*)^{(1)}
\otimes \; \cdots \; \otimes S^{a_s}(\mathfrak{u}^*)^{(s-1)}.
\end{equation}
The summation runs over all $s$-tuples $\{a_1, \dots , a_s\}$  of non-negative integers satisfying
\begin{equation}\label{U:indices-even}
i +j = \sum_{n =1}^{s} a_n  \mbox{   and     } i = \sum_{n= 1}^{s} a_n  2^{n-1}.
\end{equation}
Then we obtain a spectral sequence $E_1^{i,j} \Rightarrow \opH^{i+j}(U_s, M).$

All differentials in the above are $T$-module maps. Since $\operatorname{H}^{\bullet}(U_{s},M)^{T_{s}}\cong \operatorname{H}^{\bullet}(B_{s},M)$ (cf. \cite[I.6.9]{Jan}), it 
follows that one obtains a spectral sequence by applying $T_s$-invariants:
\begin{equation}\label{Br:spectral}
(E_1^{i,j})^{T_s} \Rightarrow \opH^{i+j}(B_s, M).
\end{equation}


\subsection{Weight calculations}\label{SS:weight} In this section, we analyze the  combinatorics for the $E_{1}$-terms of the spectral sequence defined in the previous section 
and give upper bounds on the weights appearing in certain $B_s$-cohomology spaces.  

We first define two invariants which measure the size of the weights in a rational $T$-module. 
\begin{definition} For a rational $T$-module $M,$ 
set 
\begin{itemize} 
\item[(1)]
$b(M) = \max \{\langle \sigma, {\beta}^{\vee}\rangle\; |\; \sigma \text{  a weight of } M \text{ and } \beta \text{ is a long root in } \Phi\}$
\item[(2)] $t(M) = \lceil \log_p(b(M)+1)\rceil$.
\end{itemize}
\end{definition}
For a $G$-module $M$, we have $b(M)=\max \{\langle \sigma, {\ta}^{\vee}\rangle\; |\; \sigma \text{  a weight of } M \}$. 
Given a weight $\la \in X(T)$, one can view this as a one dimensional $B$-module by letting $U$ act trivially. We denote this module simply by $\la$. If $\la$ is dominant, then $b(\la) =  \langle \la, {\ta}^{\vee}\rangle.$
 If $\la$ is dominant and nonzero, then $t=t(\la) >0$ and there exists a unique  $t$-tuple $\{\la_0, ... , \la_{t-1} \}$ with $b(\la) = \sum_{n=0}^{t-1} \la_n p^n$ and $0 \leq \la_n \leq p-1$. Moreover,  $\la_{t-1} > 0.$
 

\begin{proposition}\label{Br:weight}  Let $\la \in X(T)_+$  with $\la \neq 0$ and $\mu \in X(T)$. Set $t = t(\la)$ and  assume that $m \geq 0$, $f\geq t(\mu)$, and $s \geq t$.  Define the $t$-tuple $\{\la_0, ... , \la_{t-1} \}$ via $\langle \la, \ta^{\vee}\rangle = \sum_{n=0}^{t-1}\la_n p^n$ and $0 \leq \la_n \leq p-1$.
\begin{itemize}
\item[(a)] If $p=2$ and $\ga$ is a weight of $\opH^m(B_{s+f}, p^s\mu+\la)^{(-(s+f))},$ then 
\begin{equation}\label{Br:weight-even}
b(\ga) \leq m - (s-t).
\end{equation}
Equality can hold only if $f=t(\mu).$
\item[(b)] If $p$ is odd  and $\ga$ is a weight of $\opH^m(B_{s+f}, p^s\mu+\la)^{(-(s+f))},$ then 
\begin{equation}\label{Br:weight-odd}
b(\ga) \leq \min\{m - (s-t+1)(p-2) + \la_{t-1}, m-(s-t)(p-2)\}.
\end{equation}
Equality can hold only if $f=t(\mu).$
\end{itemize}
\end{proposition}
\begin{proof} 
(a) We make use of the spectral sequence of Section \ref{SS:spectral}. If $\ga$ is a weight of $\opH^m(B_{s+f}, 2^s\mu+\la)^{(-(s+f))}$, then 
$2^{(s+f)}\ga$ is a weight of an expression of the form 
\begin{equation}\label{Ur:weight-even}
( \la+2^s \mu) \otimes S^{a_1}(\mathfrak{u}^*)
\otimes S^{a_2}(\mathfrak{u}^*)^{(1)}
\otimes \;\cdots \; \otimes S^{a_{s+f}}(\mathfrak{u}^*)^{(s+f-1)}
\end{equation}
with  
$m = \sum_{n= 1}^{s+f} a_n.$ 

We can write $2^{s+f} \ga = \la + 2^s\mu+ \sum_{n=1}^{s+f} \xi_n 2^{n-1}$, where each $\xi_n$  is a weight of $S^{a_n}(\mathfrak{u}^*)$. 
For any long root $\beta$ and any positive root $\alpha$,
\begin{equation}\label{long-root}
\langle \alpha, \beta^{\vee}\rangle  
= 2 \text { if } \alpha=\beta \text{ while }
\langle \alpha, \beta^{\vee}\rangle \leq 
1 \text { if } \alpha\neq \beta.
\end{equation} 
This implies that, for all long roots $\beta$ and $1 \leq n \leq s + f$,
\begin{equation}\label{a-estimate-even}
\ceil{\langle \xi_n, \beta^{\vee}\rangle/2} \leq a_n. 
\end{equation} 
Note also that, since $\xi_n$ is a sum of positive roots, $\langle \xi_n, \ta^{\vee}\rangle \geq 0$.
For convenience, set $c_n := \ceil{\langle \xi_n, \ta^{\vee}\rangle/2}$. Then we have $2a_n \geq 2c_n \geq \langle \xi_n, \ta^{\vee}\rangle\geq 0$ for $1 \leq n \leq s + f$.

Since $\la$ is a dominant weight, for any long root $\beta$, we have $ \langle \la, \beta^{\vee}\rangle \leq \langle \la, \ta^{\vee}\rangle.$  It follows that 
\begin{equation}\label{ga-estimate-even}
2^{s+f} b(\ga) \leq \langle \la, \ta^{\vee}\rangle+ 2^s b(\mu)+\sum_{n=1}^{s+f}\langle \xi_n, \beta^{\vee}\rangle2^{n-1}\leq \langle \la, \ta^{\vee}\rangle+ 2^s b(\mu)+\sum_{n=1}^{s+f}a_n2^{n}.
\end{equation}
In order to find an upper bound on $b(\ga)$ we need to 
find an upper bound on $\sum_{n=1}^{s+f}a_n 2^{n}$ under the constraints
\begin{equation}\label{firstconstraint}
\sum_{n=1}^{s+f}a_n = m,  a_n \geq c_n \text{ for } 1 \leq n \leq s + f.
\end{equation}
However, the maximum value of $\sum_{n=1}^{s+f}a_n 2^{n}$ subject to 
(\ref{firstconstraint}) is less than or equal to its maximum value subject to the following (weaker) constraints
\begin{equation}\label{constraint}
\sum_{n=1}^{s+f}a_n= m,  \;
 a_n \geq  c_n \text{ for } 1 \leq n \leq s,\ a_n \geq 0 \text{ for } s+1 \leq n \leq s+f-1.
\end{equation}
In (\ref{constraint}) we are weakening the condition on some of the $a_n$s and not mentioning $a_{s+f}$ at all.  The sum $\sum_{n=1}^{s+f}a_n 2^{n}$ is evidently largest when $a_{s + f}$ is as large as possible.  This will correspond to choosing the other $a_n$s to be as small as possible.  

In other words, to maximize $\sum_{n=1}^{s+f}a_n 2^{n}$ subject to (\ref{constraint}) (and hence also (\ref{firstconstraint})), we may assume from now on that 
\begin{equation}\label{assumption-even}
a_n = c_n \text{ for } 1 \leq n \leq s  \text{ and } a_n = 0 \text{ for }s+1 \leq n \leq s + f - 1.
\end{equation}  
Observe that this implies
\begin{equation}\label{max-min-even}
\langle \xi_n, \ta^{\vee}\rangle \leq 2a_n \leq \langle \xi_n, \ta^{\vee}\rangle +1 \text{ for } 1 \leq n \leq s -1.
\end{equation}

We now investigate how large $a_{s+ f}$ can be.
For $t\leq n \leq s-1$, set $\la_n =0.$ 
Then we can write 
$$2^{s+f} \langle \ga, \ta^{\vee} \rangle =   \sum_{n=0}^{s-1}  (\la_n +\langle \xi_{n+1}, \ta^{\vee}\rangle) 2^n+2^s \langle \mu, \ta^{\vee} \rangle+ \sum_{n=s}^{s+f-1}\langle \xi_{n+1}, \ta^{\vee}\rangle2^{n}.$$ 
The inner products 
$\langle \xi_n, \ta^{\vee}\rangle$ are non-negative and at least $\la_{t-1}$ is positive. 
By divisibility considerations, this implies that
there exist non-negative integers $k_0, \dots, k_s$ with $k_0 =0$ such that 
\begin{equation}\label{k-intro-even}
 \langle \xi_{n+1}, \ta^{\vee}\rangle + \la_n= 2k_{n+1} - k_n, \text{ for } 0 \leq n \leq s-1.
\end{equation}
Then we have
\begin{align*}
\sum_{n=0}^{s-1}  (\la_n +\langle \xi_{n+1}, \ta^{\vee}\rangle) 2^n &= \sum_{n=0}^{s-1}(2k_{n+1}-k_n)2^n \\
& = 2^s k_s - 2k_0 \\
&= 2^s k_s.
\end{align*}

Note that, if $k_u = 0$ for some $u$, then $k_n =  \la_n =0$ for all $n <u$ and $ \langle \xi_{n}, \ta^{\vee}\rangle = 0$ for all $n \leq u$. 
On the other hand, $\la_{t-1}$ is positive. This forces  $k_n \geq 1$ for all $t \leq n \leq s$. Therefore there exists  a $u$ with $0 \leq u\leq t-1$ and $k_n =0$ for all $n \leq u$ and $k_n\geq 1$ for all $u+1 \leq n \leq s-1.$

For $n=u$ we obtain
$
\langle \xi_{u+1}, \ta^{\vee}\rangle = 2k_{u+1}-\la_t\geq 2k_{u+1}-1.
$
This forces 
\begin{equation}\label{t-even}
a_{u+1} = c_{u+1} = k_{u+1}.
\end{equation}
For $u+1\leq n \leq t-1$,  we conclude that 
$$
 \langle \xi_{n+1}, \ta^{\vee}\rangle = 2k_{n+1}-k_n-\la_n\geq 2k_{n+1}-k_n+(1-k_n)-\la_n\geq2(k_{n+1}-k_n),
$$
which implies 
\begin{equation}\label{n<s-even}
a_{n+1} = c_{n+1}  \geq (k_{n+1}-k_n).
\end{equation}
Finally for $t\leq n \leq s-1$, 
$$
\langle \xi_{n+1}, \ta^{\vee}\rangle = 2k_{n+1}-k_n-\la_n \geq 2k_{n+1}-k_n+(1-k_n)-\la_n =2(k_{n+1}-k_n)+1,
$$
which implies 
\begin{equation}\label{s<n-even}
a_{n+1} = c_{n+1} \geq (k_{n+1}-k_n)+1.
\end{equation}
Combining (\ref{t-even}), (\ref{n<s-even}) and (\ref{s<n-even}) yields
$\sum_{n=u+1}^{s}a_n  \geq   k_s+ s-t,$ which forces $a_{s+f} \leq m-s+t-k_s.$

Using the above inequality, (\ref{assumption-even}), and (\ref{max-min-even}), one obtains from
(\ref{ga-estimate-even}) via (\ref{k-intro-even}) and our assumptions on $a_n$
\begin{eqnarray*}
2^{s+f} b(\ga)  &\leq& \langle \la, \ta^{\vee}\rangle+ 2^s b(\mu) +\sum_{n=1}^{s+f}a_n2^n\\
&=& \sum_{n=0}^{s-1}\la_n2^n + \sum_{n=1}^{s}a_n2^n + 2^sb(\mu) + 
\sum_{n=s+1}^{s+f-1}a_n2^n + a_{s+f}2^{s + f}\\
&=& \sum_{n=0}^{s-1}(\la_n + 2a_{n+1})2^n + 2^s b(\mu) + a_{s+f}2^{s+f}\\
& \leq&  \sum_{n=0}^{s-1}( \la_n+\langle \xi_{n+1}, \ta^{\vee}\rangle+1)2^n+ 2^s(2^{t(\mu)}- 1) + (m-(s-t)-k_s) 2^{s+f}
\\
&=& \sum_{n=0}^{s-1}(\la_n + \langle \xi_{n+1},\ta^{\vee}\rangle)2^n + \sum_{n=0}^{s-1}2^n + (2^{s+{t(\mu)}} - 2^s) + (m-(s-t)-k_s)2^{s+f}\\
&=& k_s 2^s + (2^s - 1) + 2^{s+{t(\mu)}}-2^s+ (m-(s-t)-k_s) 2^{s+f}\\
&=&(m-(s-t))2^{s+f}-k_s2^{s+f}+k_s2^s+2^{s+{t(\mu)}}-1\\
&=& (m-(s-t))2^{s+f} + k_s(- 2^{s + f} + 2^s) + 2^{s + {t(\mu)}} - 1.
\end{eqnarray*}
Therefore,
\begin{equation}\label{final-even}
b(\ga) \leq (m-(s-t)) + k_s\left(-1 + 2^{- f}\right) + 2^{t(\mu)-f} - 2^{-(s + f)}.
\end{equation}
Since $k_s \geq 1$ and $f \geq t(\mu)$, one obtains 
$b(\ga) < m- (s-t)+1$ or $b(\ga) \leq m-(s-t),$ as claimed. If $f > t(\mu)$  we can conclude the stronger inequality $b(\ga) < m-(s-t).$


(b) Again we make use of the spectral sequence in the previous section. If $\ga$ is a weight of $\opH^m(B_{s+f},p^s\mu+\la)^{(-(s+f))}$ then 
$p^{s+f}\ga$ is a weight of an expression of the form 
\begin{equation}\label{Ur:weight-odd}
 (\la +p^s \mu)\; \otimes S^{a_0}(\mathfrak{u}^*)
\otimes S^{a_1}(\mathfrak{u}^*)^{(1)}
\otimes \; \cdots\; \otimes S^{a_{s+f}}(\mathfrak{u}^*)^{(s+f)}\otimes \Lambda^{b_0}(\mathfrak{u}^*) \otimes \Lambda^{b_1}(\mathfrak{u}^*)^{(1)}
\;\cdots \; \otimes \Lambda^{b_{s+f}}(\mathfrak{u}^*)^{(s+f)}
\end{equation}
with  $a_0=b_{s+f}=0$ and
$m = \sum_{n= 0}^{s+f}(2a_n +b_n).$ 

We can write $p^{s+f} \ga = \la + p^s \mu +\sum_{n=0}^{s+f} (\xi_n  + \psi_n) p^n$, where each $\xi_n$ and $\psi_n$ is a weight of $S^{a_n}(\mathfrak{u}^*)$ or $\Lambda^{b_n}(\mathfrak{u}^*)$, respectively. From 
(\ref{long-root}) it follows that, for all long roots $\beta$ and $0 \leq n \leq s + f$,
\begin{equation}\label{a-b-estimate}
\langle \xi_n + \psi_n, \beta^{\vee}\rangle \leq 2a_n + b_n + 1.
\end{equation}
We are using here that, as $\psi_n$ comes from an exterior power, $\ta$ can appear at most once in $\psi_n$.  Notice also that, if $\langle \xi_n +\psi_n, \beta^{\vee}\rangle \leq 1$, one could make the stronger claim that 
\begin{equation}\label{a-b-estimate-2}
\langle \xi_n + \psi_n,\beta^{\vee}\rangle \leq 2a_n + b_n.
\end{equation} 
It follows that 
\begin{equation}
2a_n + b_n + 1 \geq \max\{1, \langle \xi_n + \psi_n, \beta^{\vee}\rangle\} 
\end{equation} 
For convenience, set $c_n := \max\{1,\langle \xi_n + \psi_n,\ta^{\vee}\rangle\}$.  Then, we have $2a_n + b_n + 1 \geq c_n \geq 1$ for $1 \leq n\leq s+ f$.

As before we have $\langle \la, \beta^{\vee}\rangle \leq \langle \la, \ta^{\vee}\rangle$ for any long root $\beta$, which yields
\begin{equation}\label{ga-estimate-odd}
p^{s+f} b(\ga) \leq \langle \la, \ta^{\vee}\rangle+
p^s b(\mu)+ \sum_{n=0}^{s+f-1}(2a_n + b_n +1)p^n + 2a_{s+f} p^{s+f}.
\end{equation}
In order to find an upper bound for $b(\ga)$ we need an upper bound for $\sum_{n=0}^{s+f-1}(2a_n + b_n +1)p^n + 2a_{s+f} p^{s+f}$ under the constraints
\begin{equation}\label{constraints-odd} 
m+s+f = \sum_{n=0}^{s+f-1}(2a_n + b_n +1) + 2a_{s+f},\;
2a_n+ b_n +1 \geq c_n \text{ for } 0\leq n\leq s+f.
\end{equation} 
As in the $p = 2$ case, we attempt to maximize $a_{s + f}$ while minimizing the other $a_n$s. Thus it suffices to assume from now on that $2a_n + b_n + 1 = c_n$ for $0 \leq n \leq s-1$ while 
$2a_n + b_n + 1 = 1$ for $s \leq n \leq s + f - 1$.  Our goal is again to determine how large $2a_{s + f}$ can be.  Note that our assumption implies
\begin{equation}\label{max-min-odd}
\langle \xi_n + \psi_n, \ta^{\vee}\rangle \leq 
2a_n+ b_n +1\leq \langle \xi_n + \psi_n, \ta^{\vee}\rangle +1 \text{ for } 0 \leq n \leq s-1.
\end{equation}

For $t\leq n \leq s-1$ we set $\la_n =0.$ 
Then we can write 
$$p^{s+f} \langle \ga, \ta^{\vee} \rangle =   \sum_{n=0}^{s-1}  (\la_n +\langle \xi_n + \psi_n, \ta^{\vee}\rangle) p^n + p^s\langle \mu, \ta^{\vee} \rangle +
\sum_{n=s}^{s+f}  \langle \xi_n + \psi_n, \ta^{\vee}\rangle p^n.$$ The inner products 
$\langle \xi_n + \psi_n, \ta^{\vee}\rangle$ are non-negative and at least $\la_{t-1}$ is positive.  This implies that
there exist non-negative integers $k_0,\dots , k_s$ with $k_0 =0$ such that 
\begin{equation}\label{k-intro}
 \la_n+\langle \xi_n + \psi_n, \ta^{\vee}\rangle = k_{n+1}p - k_n, \text{ for } 0 \leq n \leq s-1.
\end{equation}
Then we have
$$
\sum_{n = 0}^{s - 1}(\la_n + \langle\xi_n + \psi_n,\ta^{\vee}\rangle) = k_s p^s.
$$

Note that, if there exists a $u$ with $k_u = 0$, then $k_n = \langle \xi_n + \psi_n, \ta^{\vee}\rangle=\la_n =0$ for all $n <u$. 
On the other hand, $\la_{t-1}$ is positive. This forces  $k_n \geq 1$ for all $n \geq t$. Therefore there exists  a $u$ with $0 \leq u\leq t-1$ and $k_n =0$ for all $n \leq u$ and $k_n\geq 1$ for all $n > u.$

For $u\leq n \leq t-1$ we obtain
$$
\langle \xi_n + \psi_n, \ta^{\vee}\rangle = k_{n+1}p - k_n - \la_n = (p-1)-\la_n -(k_n-1)+ (k_{n+1}-1)(p-1) +(k_{n+1} -1).$$ 
If $\la_n < p-1$ or $k_{n+1}>1$, then 
$$2a_n + b_n +1 \geq\langle \xi_n + \psi_n, \ta^{\vee}\rangle \geq 1 -(k_n-1)+ (k_{n+1}-1)(p-1).$$
If $\la_n = p-1$ and $k_{n+1}=1$, then $ \langle \xi_n + \psi_n, \ta^{\vee}\rangle = -(k_n-1)\leq 1$ and, by (\ref{a-b-estimate}),   
$$2a_n + b_n +1 \geq \langle \xi_n + \psi_n, \ta^{\vee}\rangle +1 = 1-(k_n-1).$$
 Either way one obtains, for  $u\leq n \leq t-1,$ that
\begin{equation}\label{t-s-estimate} 
2a_n + b_n +1   \geq 1 -(k_n-1)+ (k_{n+1}-1)(p-1) .
\end{equation}

For $t \leq n \leq s-1$ we have 
\begin{equation}\label{s-estimate} 
2a_n + b_n +1 \geq \langle \xi_n + \psi_n, \ta^{\vee}\rangle = k_{n+1}p-k_n=(p-1)-(k_n-1)  +(k_{n+1}-1)p.
\end{equation}

For $n \leq u-1$ and $s\leq n \leq s+f-1$, our assumptions give $2a_n + b_n +1 \geq 1$.
Combining this with (\ref{t-s-estimate}), (\ref{s-estimate}), and our assumptions above yields
\begin{align*}
&\sum_{n=0}^{s+f-1}(2a_n + b_n +1)\\
 &= \sum_{n=0}^{u-1}(2a_n + b_n +1) + \sum_{n=u}^{t-1}(2a_n + b_n +1) + \sum_{n=t}^{s-1}(2a_n + b_n +1) +  \sum_{n=s}^{s+f-1}(2a_n + b_n +1)\\
&\geq  \sum_{n = 0}^{u-1}1 + \sum_{n = u}^{t-1}(1 - (k_n - 1) + (k_{n+1} - 1)(p-1)) + 
\sum_{n = t}^{s - 1}((p-1) - (k_n - 1) + p(k_{n+1} - 1)) + \sum_{n = s}^{s + f - 1}1\\
&= u + (t - u) - \sum_{n=u}^{t-1}(k_n - 1) + (p-1)\sum_{n=u}^{t-1}(k_{n+1} - 1) + 
(s-t)(p-1) - \sum_{n=t}^{s-1}(k_n - 1) + p\sum_{n = t}^{s-1}(k_{n+1} - 1) + f\\
&= t + f + (s-t)(p-1) - \sum_{n=u}^{s-1}(k_n - 1) + (p-1)\sum_{n = u+1}^{t}(k_n - 1) + p\sum_{n = t + 1}^{s}(k_n - 1)\\
& = t + f + (s-t)(p-1) + 1 - \sum_{n = u+1}^{t}(k_n - 1) - \sum_{n = t + 1}^{s-1}(k_n - 1) + (p-1)\sum_{n=u+1}^{t}(k_n - 1)\\
&\qquad + p\sum_{n = t+1}^{s-1}(k_n - 1) + p(k_s - 1)\\
& =  1+t +f + (s-t)(p-1)+(p-2)\sum_{n = u+1}^t (k_n-1) +(p-1)\sum_{n=t+1}^{s}(k_n-1) + p(k_s-1).
\end{align*}
This expression is minimal if all $k_n=1$ for $u+1\leq n\leq s-1$.  
One obtains
\begin{equation}\label{sum-2}
\sum_{n=0}^{s+f-1}(2a_n + b_n +1) \geq k_s +t+f+ (s-t)(p-1)= k_s+ s +f+(s-t)(p-2).
\end{equation}
Hence, from (\ref{constraints-odd}), $2a_{s+f} \leq m+s+f -k_s -s-f -(s-t)(p-2)=m-k_s-(s-t)(p-2).$ 
Substituting the above and (\ref{max-min-odd}) into (\ref{ga-estimate-odd}) with $k_s =1$ yields
\begin{eqnarray*}
p^{s+f} b(\ga)  &\leq& \langle \la, \ta^{\vee}\rangle+p^s b(\mu)+ \sum_{n=0}^{s+f-1}(2a_n+b_n+1)p^n + 2a_{s+f}p^{s+f} \\
&=& \sum_{n=0}^{s-1}\la_np^n + \sum_{n=0}^{s-1}(2a_n + b_n + 1)p^n + p^s b(\mu)+ \sum_{n=s}^{s+f-1}(2a_n+b_n+1)p^n + 2a_{s+f}p^{s+f}\\
&\leq&  \sum_{n=0}^{s-1}( \la_n+\langle \xi_n+\psi_n, \ta^{\vee}\rangle+1)p^n+ p^s(p^{t(\mu)} - 1) + \sum_{n=s}^{s+f-1}p^n+2a_{s+f} p^{s+f}\\
&\leq& k_s p^s + \sum_{n=0}^{s-1}p^n + p^{s + {t(\mu)} } - p^s + \sum_{n=s}^{s + f - 1}p^n + (m - k_s - (s-t)(p-2))p^{s + f}\\
&\leq& k_s (p^s-p^{s+f}) + \sum_{n=0}^{s+f-1}p^n+p^{s+{t(\mu)} }-p^s+ (m-(s-t)(p-2)) p^{s+f}\\
&\leq&(m-(s-t)(p-2))p^{s+f}+\frac{p^{s+f}-1}{p-1}-p^{s+f}+p^{s+{t(\mu)} }.
\end{eqnarray*}
It follows that 
\begin{equation}
 b(\ga) <    m-(s-t)(p-2)+\frac{1}{p-1}-1+p^{{t(\mu)} -f}.
 \end{equation} 
Hence, 
\begin{equation}\label{last-estimate}
 b(\ga) \leq   m-(s-t)(p-2)
 \end{equation}
 with equality possible only if $f= t(\mu)$. 
 
If $\la_{t-1} < (p-2)$ we can use the estimate $2a_{t-1}+b_t + 1\geq (p-1) - (k_{t-1}-1)-\la_{t-1}$ instead of $2a_{t-1}+b_t + 1\geq 1- (k_{t-1}-1)$ and obtain $2a_{s+f} \leq k_s+s+(s-t+1)(p-2)+ \la_{t-1}$, which results in
\begin{equation}
b(\ga) \leq    m-(s-t+1)(p-2) +\la_{t-1},
\end{equation}
as claimed.  
\end{proof}


\subsection{}  Assume that $m \geq 0$, $s \geq 0$, and $f\geq 0$.
If we drop the conditions on the sizes of $\la$ and $\mu$ in Proposition \ref{Br:weight} and simply require $\la, \mu \in X(T)$, then the constraints on the non-negative integers $a_n$ and $b_n$, as defined in  (\ref{Ur:weight-even}) and (\ref{Ur:weight-odd}), are simply 
$\sum_{n=1}^{s+f} a_n=m$, if $p=2$, and $\sum_{n=0}^{s+f}(2a_n + b_n) =m$, if $p$ is odd. Obviously one obtains an upper bound for the weights of $\opH^m(B_{s+f}, p^r\mu+ \la)$  if one sets $a_{s+ f} =m$ and $2a_{s + f} =m$, respectively. 

\begin{proposition} \label{proposition:Br-remark}Assume that $m \geq 0$, $s\geq 0$, and $f\geq 0$. Let $\la, \mu \in X(T)$.
\begin{itemize}
\item[(a)] If $p=2$  and $\ga$ is a weight of $\opH^m(B_{s+f}, p^s\mu+\la)^{(-(s+f))},$ then 
\begin{equation}\label{Br:weight-even-rough}
2^{s + f}b(\ga) \leq 2^s b(\mu)+b(\la) +m\cdot 2^{s + f}.
\end{equation}
\item[(b)] If $p$ is odd  and $\ga$ is a weight of $\opH^m(B_{s+f}, p^s\mu+\la)^{(-(s+f))},$ then 
\begin{equation}\label{Br:weight-odd-rough}
p^{s + f}b(\ga) \leq p^s b(\mu)+b(\la) +m\cdot p^{s + f}.
\end{equation}
\end{itemize}
\end{proposition}

\subsection{$B_s$-vanishing}   We can now state and prove sufficient conditions for the vanishing of the cohomology group $\operatorname{H}^{m}(B_{s},\lambda)$ 
in terms of an interrelationship between $s$, $m$, and $\lambda$. 

\begin{proposition}\label{Br:vanishing}  Assume that $m \geq 0$ and $s \geq 0$. Let $\la \in X(T)$ with $\langle \la, \ta^{\vee}\rangle > 0.$ 
\begin{itemize}
\item[(a)] If $p=2$ and $s\geq m+ \ceil{\log_2(\langle \la, \tilde{\alpha}^{\vee}\rangle+1) }$, then $\opH^m(B_s,  \la) =0$.
\item[(b)] If $p$ is odd and $s \geq  m/(p-2) + \ceil{\log_p(\langle \la, \tilde{\alpha}^{\vee}\rangle+1)}$, then $\opH^m(B_s, \la) =0$.
\item[(c)] 
Set $t = \ceil{\log_p(\langle \la, \tilde{\alpha}^{\vee}\rangle+1)}$ and  define the $t$-tuple $\{\la_0, ... , \la_{t-1} \}$ via $\langle \la, \ta^{\vee}\rangle = \sum_{n=0}^{t-1}\la_n p^n$ and $0 \leq \la_n \leq p-1$. If $p$ is odd and $s\geq  m/(p-2) + \ceil{\log_p(\langle \la, \tilde{\alpha}^{\vee}\rangle+1)} +\left( \frac{\la_{t-1}}{p-2}-1 \right)$, then $\opH^m(B_s, \la) =0$.
\end{itemize}
\end{proposition}
\begin{proof} 
We prove the contrapositive. Assume that  $\opH^m(B_s,  \la) \neq 0$ and let $p^{s}\ga$ be a weight of $\opH^m(B_s,  \la).$ It follows from 
(\ref{Ur:weight-even}) and (\ref{Ur:weight-odd}) (with $\mu =0$) that $\langle \ga, \ta^{\vee} \rangle \geq \langle \la, \ta^{\vee} \rangle >0.$ 
Setting  $t=\ceil{\log_p(\langle \la, \tilde{\alpha}^{\vee}\rangle+1)}$  is equivalent to saying $p^{t-1} \leq \langle \la, \ta^{\vee} \rangle < p^t.$ We appeal to the proof of Proposition \ref{Br:weight} with $\mu =0$, 
$f=0$ and $\beta = \ta$, and obtain the following inequalities:
$$1 \leq \langle \ga, \ta^{\vee} \rangle \leq m-(s-t) \text{ if } p=2,
\text{ and }
1 \leq \langle \ga, \ta^{\vee} \rangle \leq m-(s-t)(p-2) \text{ if } p>2.$$
It follows that  $0<m-s+t$ for $p=2$, and $0 < m-(s-t)(p-2)$ for odd $p$.
Solving for $s$ yields
$s < m+t$  and $s < m/(p-2)+t$, respectively.  This verifies assertions (a) and (b). 

For part (c) we use   the inequality  $1 \leq \langle \ga, \ta^{\vee} \rangle \leq m-(s-t+1)(p-2) + \la_{t-1}.$ Again solving for $s$ yields $s < m/(p-2) + t +\left(\frac{\la_{t-1}}{p-2}-1\right)$ as desired.

Note that it is not necessary for $\la$ to be dominant, because we are only using inner products with the longest root $\ta$.
\end{proof}


\section{$B$- and $G$-cohomology}

\subsection{Vanishing of $B$-cohomology}The computation of the $B$-cohomology groups $\text{H}^{n}(B,\sigma)$, $\sigma\in X(T)$, is still an outstanding open problem. For $n=0,1,2$ these groups 
have been computed in \cite{A}, \cite{BNP7}, \cite{W}, and for $n=3$ and $p>h$ in \cite{AR}. The results in the preceding section will now be used to provide conditions to 
ensure that these $B$-cohomology groups vanish. 
 
\begin{proposition}\label{B:vanishing}  Assume that $m \geq 0$ and $s \geq 0$. Let $\la, \mu \in X(T)$ with $\langle \la, \ta^{\vee}\rangle > 0.$ 
\begin{itemize}
\item[(a)] If $p=2$ and $s\geq m+ \ceil{\log_2(\langle \la, \tilde{\alpha}^{\vee}\rangle+1) }$, then $\opH^m(B, \la - 2^s \mu) =0$.
\item[(b)] If $p$ is odd and $s \geq m/(p-2)+ \ceil{\log_p(\langle \la, \tilde{\alpha}^{\vee}\rangle+1)}$, then $\opH^m(B, \la - p^s \mu) =0$.
\end{itemize}
\end{proposition}

\begin{proof} 
We apply the Lyndon-Hochschild-Serre (LHS) spectral sequence  
$$E^{i,j}_2= \opH^i(B/B_s, \opH^j(B_{s}, \la) \otimes -p^s\mu) \Rightarrow
\opH^{i+j}(B,\la-p^s\mu).
$$ 
Let $m = i+j$. By Proposition \ref{Br:vanishing}
all the cohomology groups $\opH^j(B_{s}, \la)$ vanish. Hence  $E^{i,j}_2=0$ for all $i+j=m$, which verifies the assertion. 
\end{proof}

\subsection{$G$-cohomology} \label{section:G-cohovanishing}  We now use Proposition \ref{B:vanishing} to obtain conditions on the vanishing of certain $G$-cohomology groups. 

\begin{theorem}\label{G:Vanishing}  Suppose that $m \geq 0$ and $s \geq 0$. Let $M$ be a finite dimensional $G$-module and $\la \in X(T)_+$ with $ \la \neq 0.$ 
\begin{itemize}
\item[(a)] If $p=2$ and $s \geq m+ \ceil{\log_2(\langle \la, \tilde{\alpha}^{\vee}\rangle+1)} $, then $\Ext_G^m(M^{(s)},H^0( \la  ) )=0.$
\item[(b)]  If $p$ is odd and $s \geq m/(p-2)+ \ceil{\log_p(\langle \la, \tilde{\alpha}^{\vee}\rangle+1)}$, then $\Ext_G^m(M^{(s)},H^0( \la  ) )=0.$
\end{itemize}
\end{theorem}

\begin{proof} Assume that $\Ext_G^m(M^{(s)},H^0( \la  )  )\neq 0$, and note that $\Ext_G^m(M^{(s)},H^0( \la  ) )\cong \Ext_B^m(M^{(s)}, \la  ).$ There has to be a weight $\mu$ of $M$ such that  
$ \Ext_B^m(p^s\mu, \la  )\cong \opH^m(B, \la -p^s\mu )\neq 0.$ Any non-zero dominant weight $\la$ satisfies $\langle \la, \ta^{\vee}\rangle > 0.$ Parts (a) and (b) now follow immediately from Proposition~\ref{B:vanishing}.
\end{proof}

We observe that in general Theorem~\ref{G:Vanishing} is false when $\la = 0$. In this case Cline, Parshall and Scott \cite{CPS} have shown that there are natural inclusions obtained from the Frobenius morphism 
\begin{equation}\label{inclusion-generic}
\Ext^m_G(M, k)\into \Ext^m_G(M^{(1)}, k) \into \Ext^m_G(M^{(2)}, k) \into \cdots \; .
\end{equation}
These eventually stabilize to the so-called generic cohomology \cite{CPSvdK}.

\begin{example}
The vanishing results observed in Theorem~\ref{G:Vanishing} do not hold in general if one replaces $H^0(\la)$ by a simple module $L(\la)$. 
Let $G= SL_2$ with $p$ odd (and recall the notation for weights given in Section \ref{SS:notation}). One has $\Ext_G^2(L(2)^{(s)}, k) \cong k$ for all $s>0$. The short exact sequence
$$0 \to L(2p-2) \to H^0(2p-2) \to k \to 0$$ 
yields a long exact sequence in cohomology, part of which is $$ \Ext_G^2( L(2)^{(s)}, H^0(2p-2)) \to 
\Ext_G^2( L(2)^{(s)}, k) \to \Ext_G^3( L(2)^{(s)}, L(2p-2)) \to 
\Ext_G^3( L(2)^{(s)}, H^0(2p-2)).$$ For large $s$ it follows from  Theorem~\ref{G:Vanishing} that 
$$\Ext_G^3( L(2)^{(s)}, L(2p-2)) \cong \Ext_G^2( L(2)^{(s)}, k) \cong k \neq 0.$$
\end{example}


\section{$G_s$-cohomology}


\subsection{} 
Next we provide a constraint on the size of the highest weights of the irreducible $G$-modules that can appear as composition factors of certain $G_s$-cohomology groups.  
Proposition \ref{Gr:weight} is a main ingredient in the proof of the key $G$-cohomological vanishing result (Proposition \ref{vanishing}).

\begin{proposition}\label{Gr:weight}  Let $\la\in X(T)_+$ with $\la \neq 0$ and 
$M$ be a finite dimensional $G$-module. Set $t = t(\la)$ and assume that $m \geq 0$, $f\geq t(M)$, and $s \geq t.$ We define the $t$-tuple $\{\la_0, ... , \la_{t-1} \}$ via 
$\langle \la, \ta^{\vee}\rangle = \sum_{n=0}^{t-1}\la_n p^n$ and $0 \leq \la_n \leq p-1$.
\begin{itemize}
\item[(a)] If $p=2$ and $L(\ga)$ is a $G$-composition factor of $\opH^m(G_{s+f}, M^{(s)}\otimes H^0(\la))^{(-(s+f))}$, then 
\begin{equation}\label{Gr:weight-even}
\langle \ga, \ta^{\vee}\rangle \leq m - (s-t).
\end{equation}

Equality can hold only if $f=t(M).$
\item[(b)] If $p$ is odd and $L(\ga)$ is a $G$-composition factor of $\opH^m(G_{s+f}, M^{(s)}\otimes H^0(\la))^{(-(s+f))}$, then
\begin{equation}\label{Gr:weight-odd}
\langle \ga, \ta^{\vee}\rangle \leq \min\{m - (s-t+1)(p-2) + \la_{t-1},\ m-(s-t)(p-2)\}.
\end{equation}
Equality can hold only if $f=t(M).$
\end{itemize}
\end{proposition}
\begin{proof} We make use of the spectral sequence \cite[II.12.2]{Jan}
$$E^{i,j}_2= R^i\ind_B^G(\opH^j(B_{s+f},M^{(s)}\otimes \la)^{(-(s+f))}) \Rightarrow
\opH^{i+j}(G_{s+f},M^{(s)}\otimes H^0(\la))^{(-(s+f))}.
$$ 
Any composition factor $L(\ga)$ of 
$\opH^{m}(G_{s+f},M^{(s)}\otimes H^0(\la))^{(-(s+f))}$ has to come from some term of the form  
$R^i\ind_B^G(\opH^j(B_{s+f},M^{(s)}\otimes \la)^{(-(s+f))})$ with $i+j=m.$ Therefore, there exists a weight $\delta$ of $\opH^j(B_{s+f},M^{(s)}\otimes \la)^{(-(s+f))}$ with $L(\ga)$ being  a composition factor of $R^i\ind_B^G(\delta)$.
By the Strong Linkage Principle \cite[II.6.13]{Jan}, there exists a $w \in W$ with $\ga \uparrow w \cdot \delta.$ 
Here $w\in W$ such that $w(\delta+\rho)\in X(T)_{+}$. It follows that
$$\langle \ga, \ta^{\vee} \rangle \leq  \langle w \cdot \delta, \ta^{\vee} \rangle \leq \langle w ( \delta), \ta^{\vee} \rangle \leq \langle  \delta, w^{-1}(\ta)^{\vee} \rangle.$$
Since  $\delta$ is a weight of $\opH^j(B_{s+f},M^{(s)}\otimes \la)^{(-(s+f))}$, there exists a weight $\mu$ of $M$ with 
$\delta$ being a weight of $\opH^j(B_{s+f}, \la+p^s\mu)^{(-(s+f))}.$ Clearly,
$\langle \mu, \beta^{\vee} \rangle \leq b(M) \leq p^f$ for all long roots $\beta$. 
The assertion follows from (\ref{Br:weight-even}) and (\ref{Br:weight-odd}) applied to $\delta$. 
\end{proof}

\subsection{} \label{Gr:remark} Just as in Proposition \ref{proposition:Br-remark}, if we drop the conditions on the size of $\la$ and weights of $M$ and simply require $\la \in X(T)_+$, then one obtains the following result. 

\begin{proposition}\label{Gr:weight-general} Assume that $m \geq 0$, $s\geq 0$, and $f\geq 0$. Let $\la \in X(T)_{+}$ and $M$ be a finite dimensional $G$-module.
\begin{itemize}
\item[(a)]  If $p=2$ and $L(\ga)$ is a $G$-composition factor of $\opH^m(G_{s+f}, M^{(s)}\otimes H^0(\la))^{(-(s+f))}$, then
\begin{equation}\label{Gr:weight-even-rough}
2^{s+f}\langle \ga, \ta^{\vee}\rangle \leq 2^sb(M) + \langle \la, \ta^{\vee}\rangle +m2^{s+f}.
\end{equation}
\item[(b)] If $p$ is odd and $L(\ga)$ is a $G$-composition factor of $\opH^m(G_{s+f}, M^{(s)}\otimes H^0(\la))^{(-(s+f))}$, then 
\begin{equation}\label{Gr:weight-odd-rough}
p^{s+f}\langle \ga, \ta^{\vee}\rangle  \leq p^sb(M) + \langle\la, \ta^{\vee}\rangle +mp^{s+f}.
\end{equation}
\end{itemize}
\end{proposition}

\section{Rational Stability}\label{rat}
\subsection{} In this section, we demonstrate that the vanishing results of  Section~\ref{section:G-cohovanishing} allow for a new proof that, for a fixed $n$, $\operatorname{H}^{n}(G,M^{(s)})$ stabilizes for $s$ sufficiently 
large. This approach to rational stability does not utilize the interplay between the generic cohomology and the cohomology for the corresponding finite groups of Lie type observed in  \cite{CPSvdK}. 
 Rather we make use of the Frobenius kernels via the Lyndon-Hochschild-Serre spectral sequence and the vanishing results of Theorem~\ref{G:Vanishing}. The first spectral sequence used in the proof makes use of the Steinberg representation. It was used in \cite{CPS} to prove the inclusions in (\ref{inclusion-generic}). 

\begin{theorem} \label{theorem:generic1}  Let $M$ be a finite-dimensional rational $G$-module, and let $m$ be a fixed integer where $m\geq 0$. 
There exists $C$ such that for $s\geq C$, $\operatorname{H}^{m}(G,M^{(s)})\cong \operatorname{H}^{m}(G,M^{(s+1)})$. 
\end{theorem} 

\begin{proof} Let $\text{St}_{1}=L((p-1)\rho)$ be the Steinberg module for $G_{1}$. Consider the LHS spectral sequence, 
$$E_{2}^{i,j}=\text{H}^{i}(G/G_{1},\text{H}^{j}(G_{1},M^{(s+1)}\otimes \text{St}_{1}\otimes \text{St}_{1}^{*})\Rightarrow 
\text{H}^{i+j}(G,M^{(s+1)}\otimes \text{St}_{1}\otimes \text{St}_{1}^{*}).$$ 
Since $\text{St}_{1}$ is projective for $G_{1}$, this spectral sequence collapses and yields 
\begin{equation}\label{Stiso}
\text{H}^{m}(G,M^{(s)})\cong 
\text{H}^{m}(G/G_{1},\text{H}^{0}(G_{1},\text{St}_{1}\otimes \text{St}_{1}^{*}\otimes M^{(s+1)})
\cong \text{H}^{m}(G,M^{(s+1)}\otimes \text{St}_{1}\otimes \text{St}_{1}^{*}).
\end{equation}
There exists a short exact sequence of $G$-modules: 
$$0\rightarrow k\rightarrow \text{St}_{1}\otimes \text{St}_{1}^{*} \rightarrow Q\rightarrow 0$$ 
where $Q$ has a good filtration with factors of the form $H^{0}(\sigma)$, $\sigma\neq 0$ and 
$\sigma \leq 2(p-1)\rho$. This short exact sequence induces a long exact sequence in cohomology.  Using the isomorphism in (\ref{Stiso}), this sequence becomes
$$\cdots \rightarrow \text{H}^{m-1}(G,M^{(s+1)}\otimes Q)\rightarrow  \text{H}^{m}(G,M^{(s+1)})\rightarrow  \text{H}^{m}(G,M^{(s)})\rightarrow \text{H}^{m}(G,M^{(s+1)}\otimes Q)\rightarrow \cdots.$$ 
Now we apply Theorem \ref{G:Vanishing}.  For $s+1\geq \frac{m }{p-2}+\lceil\text{log}_{p}(\langle 2(p-1)\rho,\tilde{\alpha}^{\vee} \rangle +1)\rceil$ for $p$-odd 
(resp. $s+1\geq m+\lceil \text{log}_{2}(\langle 2\rho,\tilde{\alpha}^{\vee} \rangle +1)\rceil$ for $p = 2$), we have  $\text{H}^{m-1}(G,M^{(s+1)}\otimes H^{0}(\sigma))=0$ and
$\opH^{m}(G,M^{(s+1)}\otimes H^0(\sigma)) = 0$
for all good filtration factors $H^0(\sigma)$ in $Q$. Hence, $\opH^{m-1}(G,M^{(s + 1)}\otimes Q) = 0$ and $\opH^{m}(G,M^{(s+1)}\otimes Q) = 0$.  Therefore, there exists $C$ (depending on $m$) 
such that, for $s\geq C$, $\operatorname{H}^{m}(G,M^{(s)})\cong \operatorname{H}^{m}(G,M^{(s+1)})$. 
Note that we can choose $C=\frac{m}{p-2}+\lceil\text{log}_{p}(2(p-1)(h^{\vee}-1)+1)\rceil-1$ for $p$-odd and 
$C = m+\lceil\text{log}_{2}(2(h^{\vee}-1)+1)\rceil-1$ for $p=2$, where $h^{\vee}$ is the dual Coxeter number. 
\end{proof}


\subsection{} In a second approach to rational stability, we will show that the structure 
of the cohomology of $G_{1}$ dictates bounds on the stability for  $\operatorname{H}^{m}(G,M^{(s)})$. 
In particular, when $\operatorname{H}^{\bullet}(G_{1},k)$ has a good filtration, the next result shows that the 
stability bounds on $\operatorname{H}^{m}(G,M^{(s)})$ can be significantly improved from those given in the proof of 
Theorem~\ref{theorem:generic1}. 

\begin{theorem} \label{theorem:generic2} Let $M$ be a finite-dimensional rational $G$-module. Suppose that $\operatorname{H}^{n}(G_{1},k)^{(-1)}$ has a good filtration for $n \leq m$. Set 
$$
F(m)=
\begin{cases} 
m & p=2\\
0 & m \leq 1 \mbox{ and } p  \mbox{ odd}\\
\frac{m}{p-2} & m >1 \mbox{ and } p  \mbox{ odd.} 
\end{cases} 
$$
For $s\geq F(m)$ we have 
$\operatorname{H}^{n}(G,M^{(s)})\cong \operatorname{H}^{n}(G,M^{(s+1)})$ for $n\leq m.$
\end{theorem} 

\begin{proof} Consider the LHS spectral sequence 
$$E_2^{i,j}=\text{H}^{i}(G/G_{1},\text{H}^{j}(G_{1},M^{(s+1)})\Rightarrow \text{H}^{i+j}(G,M^{(s+1)}).$$ 
When one untwists, we have 
$$E_{2}^{i,j}=\text{H}^{i}(G,\text{H}^{j}(G_{1},k)^{(-1)}\otimes M^{(s)}).$$ 
We will show that $E_{2}^{i,j}=0$ whenever $i+j \leq m$ and $1 \leq j.$
That will yield the desired isomorphisms  
$$\operatorname{H}^{n}(G,M^{(s)})\cong \operatorname{H}^{n}(G/G_1,\operatorname{Hom}_{G_1}(k,M^{(s+1)})) 
= E_2^{n,0} \cong \operatorname{H}^{n}(G,M^{(s+1)}) \text{ for } n \leq m.$$

By the good filtration assumption, it suffices to show that $\opH^i(G,H^0(\la)\otimes M^{(s)}) = \opH^i(B,\la\otimes M^{(s)}) = 0$ for all filtration factors $H^0(\la)$ of
$\opH^j(G_1,k)^{(-1)}$ (with $i + j \leq m$ and $1 \leq j$).  
It follows from Proposition \ref{Gr:weight-general} (with $\la = 0$, $M = k$, and $\ga = \la$) that, for any such $\la$, we have $0 < \langle \la, \tilde{\alpha}\rangle \leq j$. We apply the Lyndon-Hochschild-Serre (LHS) spectral sequence  
$$E^{i-b,b}_2= \opH^{i-b}(B/B_s, \opH^b(B_{s}, \la) \otimes M^{(s)}) \Rightarrow
\opH^{i}(B,\la\otimes M^{(s)}).
$$ 
It now suffices to show that $\opH^b(B_s,\la) = 0$ for $0 \leq b \leq i$.  

Set $t = \lceil \log_p(\langle \la, \ta \rangle +1)\rceil$ and recall (from Section \ref{SS:weight}) the definition of $\la_{t-1}$. Then 
$$p^{t-1} \leq p^{t-1} \la_{t-1} \leq \langle \la, \ta \rangle \leq j.$$ 

We will first discuss the case $p=2$. Note that, for all positive integers $x$, $x\leq 2^{x-1}.$ This implies that $t \leq j$. By Proposition \ref{Br:vanishing},
 $\opH^b(B_{s}, \la)$ vanishes as long as $s \geq b +t$.  But $b + t \leq i + j \leq m$, and we are assuming that  $s \geq m$. 
 
 If $p $ is odd we apply Proposition \ref{Br:vanishing}(c). We need to verify that  
 $$s\geq  b/(p-2) + t +\left( \frac{\la_{t-1}}{p-2}-1 \right).$$ 
But
$$b/(p-2) + t +\left( \frac{\la_{t-1}}{p-2}-1 \right)\leq m/(p-2) - \frac{p^{t-1}-1}{p-2} \la_{t-1}+t-1 \leq m/(p-2) - \frac{p^{t-1}-1}{p-2}+t-1.$$ 
The assertion follows now from the fact that $p^{x-1}-1 \geq (x-1)(p-2)$ for all positive integers $x$ and the assumption $s \geq m/(p-2)$.

Note that $\opH^1(G_1,k)=0$ when $p$ is odd \cite[II.12.2]{Jan}. This implies that 
$$E_{2}^{0,1}=\text{Hom}_G(k,\text{H}^{1}(G_{1},k)^{(-1)}\otimes M^{(s)})=0$$ 
for $s\geq 0.$ Hence, $s\geq 0$ suffices in the case $m\leq1$ and $p$ odd.

\end{proof}

\begin{corollary} \label{corollary:generic2}  If $\operatorname{H}^{n}(G_{1},k)^{(-1)}$ has a good filtration for $n \leq m$, then 
$$\operatorname{H}^{m}(G,M^{(s)})\cong \operatorname{H}^{m}(G,M^{(s+1)})$$
for $s\geq m$. 
\end{corollary} 

\begin{proof} One can directly apply Theorem~\ref{theorem:generic2} and observe that, for $p$ arbitrary, $m\geq F(m)$. 
\end{proof} 

\begin{remark} Let $G = SL_2$ and $p = 2$. In this case, it is known that $\opH^n(G_1,k)^{-1}$ admits a good filtration.   For the module $M = L(1)$, Stewart \cite[Remark 2.5]{S} showed that for $s < m$, $\dim\opH^m(G,M^{(s)}) < \dim\opH^m(G,M^{(s + 1)})$ (and that the dimensions agree for $s \geq m$).   This shows that for $p = 2$, the bound on $s$ in Corollary \ref{corollary:generic2} is sharp.
\end{remark}

In Section \ref{SS:generic} (see Remark \ref{R:genericproof}), we will give an alternative proof of Theorem \ref{theorem:generic2} (and Corollary \ref{corollary:generic2})  that uses the relationship with the cohomology of the finite groups in the spirit of \cite{CPSvdK}. The assumption that $\operatorname{H}^{n}(G_{1},k)^{(-1)}$ has a good filtration for $n \leq m$ can then be dropped.

\subsection{} In many situations it is known that $\operatorname{H}^{\bullet}(G_{1},k)^{(-1)}$ has a good filtration. One 
such case is when $p>h$ where $\operatorname{H}^{2\bullet+1}(G_{1},k)^{(-1)}=0$ and 
$\operatorname{H}^{2\bullet}(G_{1},k)^{(-1)}\cong k[{\mathcal N}]$, where ${\mathcal N}$ is 
the nilpotent cone of ${\mathfrak g}$. When $p$ is arbitrary, $\operatorname{H}^{n}(G_{s},H^{0}(\lambda))$ has a good filtration for $\lambda\in X(T)_{+}$ and $n=0,1,2,$ 
(cf. \cite{BNP4, BNP7, Jan1, W}). It is suspected that this should hold for arbitrary $n$. In fact, if $p>h$, $\operatorname{H}^{n}(G_{1},H^{0}(\lambda))$ has a good filtration for all $n$ 
\cite{KLT}. We can apply Theorem~\ref{theorem:generic2} in these aforementioned contexts. 

\begin{corollary} Let $p\geq h-1$ with $p$ odd. Then $\operatorname{H}^{m}(G,M^{(s)})\cong \operatorname{H}^{m}(G,M^{(s+1)})$
for $s\geq m/(p-2)$. 
\end{corollary} 

\begin{proof} The cohomology ring $\operatorname{H}^{\bullet}(G_{1},k)^{(-1)}$ has a good filtration for $p>h$ \cite[3.7 Corollary]{AJ}, 
$p=h$ \cite[6.3 Corollary]{AJ}, and $p=h-1$ \cite[Corollary 7.4.1]{BNPP}, \cite[6.10 Corollary]{AJ}. The result now follows from Theorem~\ref{theorem:generic2}.
\end{proof} 


\begin{corollary} \label{corollary:generic3} Let $p$ be arbitrary, and $m=0,1,2$. Then  $\operatorname{H}^{m}(G,M^{(s)})\cong \operatorname{H}^{m}(G,M^{(s+1)})$ in the 
following cases: 
\begin{itemize} 
\item[(a)] $s\geq m$ when $p=2$,
\item[(b)] $s\geq 0$ when $p$ is odd and $m \leq 1$, 
\item[(c)] $s\geq 2/(p-2)$ when $p$ is odd and $m=2$.
\end{itemize}  
\end{corollary}

We note that the preceding corollaries can be viewed as generalizations of the results in \cite[(7.1) Theorem, (7.2) Theorem] {CPSvdK}.



\section{Vanishing of certain $G$-cohomology groups}
\subsection{} The following combinatorial lemma will be used in the proof of Proposition \ref{vanishing}. Proposition \ref{vanishing}
is a stronger version of $G$-cohomology vanishing as compared to Theorem~\ref{G:Vanishing} that will enable us to connect 
rational stability with the cohomology of finite groups of Lie type. 

\begin{lemma}\label{L:numerical} Let $f \geq 1$, $s \geq 1$, and $t \geq 1$ be integers.
\begin{itemize}
\item[(a)] If $2^{t-1} \leq \frac{2^{s+f-1}-2^s}{2^{s+f}-1} + \frac{2^{s+f}}{2^{s+f} - 1}\cdot s$, then $s\geq t$.
\item[(b)] Let $p$ be an odd prime. If $p^{t-1} \leq \frac{p^{s+f-1} - p^s}{p^{s+f} - 1} + \frac{p^{s+f}}{p^{s+f} - 1}\cdot s(p-2) $, then $s \geq t$.
\item[(c)] Let $p$ be an odd prime. If $p^{t-1} \leq \frac{p^{s+f} - p^s}{p^{s+f} - 1} + \frac{p^{s+f}}{p^{s+f} - 1}\cdot(s(p-2) - 1)$, then $s \geq t$.
\end{itemize}
\end{lemma}

\begin{proof} For the proof of parts (a) and (b) we make use of  $$\frac{p^{s+f-1}-p^s}{p^{s+f}-1} \leq \frac{p^{s+f-1}-p^s+1}{p^{s+f}}< \frac{1}{p}.$$ 
(a) Using the assumption, we have
\begin{align*}
2^{t-1} &< \frac{1}{2} + \frac{2^{s+f}}{2^{s+f} - 1}\cdot s \leq  \frac{1}{2} + \frac{2^{s+1}}{2^{s+1} - 1}\cdot s = \frac{1}{2} +  \frac{2s}{2^{s+1} - 1}\cdot 2^s\\
&\leq  \left(\frac{1}{2^{s+1}} +  \frac{2s+1}{2^{s+1} }\right)\cdot 2^s 
 =  \left( \frac{2(s+1)}{2^{s+1} }\right)\cdot 2^s  =  \left( \frac{s+1}{2^{s} }\right)\cdot 2^s .
\end{align*}
It follows that (using $s \geq 1$)
$$
2^{t - s - 1} < \frac{s +1}{2^s } \leq 1.
$$
Thus $t - s - 1 < 0$, so $t < s + 1$ or $t \leq s$ as claimed.

(b) Using the assumption, we have
\begin{align*}
p^{t-1} &\leq \frac{1}{p} + \frac{p^{s+f}}{p^{s+f} - 1}\cdot s(p-2)\leq \frac{1}{p} + \frac{s(p-2)}{p^{s+1} - 1}\cdot p^{s+1} < \frac{1}{p} + \frac{s(p-2)+1}{p^{s+1} }\cdot p^{s+1}\\ &\leq \frac{1}{p} + \frac{p(s(p-2)+s)}{p^{s+1} }\cdot p^{s}
=  \frac{ps(p-1)+1}{p^{s+1} }\cdot p^{s} <  \frac{p^2s}{p^{s+1} }\cdot p^{s} =\frac{s}{p^{s-1} }\cdot p^{s}.
\end{align*}
Therefore, 
$$
p^{t-s-1} < \frac{s}{p^{s - 1}} \leq 1.
$$
Hence, $t - s - 1 < 0$ or $t \leq s$. 

(c) Using the assumption, we have
\begin{align*}
p^{t-1} &\leq \frac{p^{s+f} - p^s}{p^{s+f} - 1} + \frac{p^{s+f}}{p^{s+f} - 1}\cdot(s(p-2) - 1)
< \frac{p^{s+f}}{p^{s+f} - 1} + \frac{p^{s+f}}{p^{s+f} - 1}\cdot(s(p-2) - 1)\\
&= \frac{p^{s+f}}{p^{s+f}-1}\cdot s(p-2) \leq \frac{p^s}{p^s - 1}\cdot s(p-2).
\end{align*}
Therefore, 
$$
p^{t-s-1} < \frac{s(p-2)}{p^s - 1} = \frac{s}{1 + p + p^2 + \cdots + p^{s-1}}\cdot\frac{p-2}{p-1} < 1\cdot 1 = 1.
$$
Hence, $t - s - 1 < 0$ or $t \leq s$. 
\end{proof}

\subsection{} The cohomological vanishing result proved below will enable us to provide sufficient conditions for the vanishing of 
terms (in the third column) of the long exact sequence (\ref{eq:lesrestriction}).  

\begin{proposition}\label{vanishing}  Assume that $m \geq 0$. Let $M$ be a finite dimensional $G$-module and $\la \in X(T)_+$ with $ \la \neq 0.$ 
\begin{itemize}
\item[(a)] If $p=2$, $s \geq m$ and $f \geq  \ceil{\log_2 (b(M)+1)}+1$, then $\Ext_G^m(V(\la)^{(s+f)}, M^{(s)} \otimes H^0( \la  ) )=0.$

\item[(b)] If $p$ is odd, $s \geq m/(p-2)$ and $s+f \geq  \lfloor{m/(p-2)}\rfloor +\ceil{\log_p (b(M)+1)}+1 $, then $\Ext_G^m(V(\la)^{(s+f)}, M^{(s)} \otimes H^0( \la  ) )=0.$
\end{itemize}
\end{proposition}

\begin{proof}  Assume that $\Ext_G^m(V(\la)^{(s+f)}, M^{(s)} \otimes H^0( \la  ) )\neq 0.$
We apply the LHS spectral sequence 
$$E^{i,j}_2= \Ext^i_{G/G_{s+f}}(V(\la)^{(s+f)}, \opH^j(G_{s+f}, M^{(s)} \otimes H^0( \la  ) ) \Rightarrow
\Ext_G^m(V(\la)^{(s+f)}, M^{(s)} \otimes H^0( \la  ) ).
$$ 
By assumption, there exist $i, j$ with $i + j = m$ such that 
$$\Ext^i_{G/G_{s+f}}(V(\la)^{(s+f)},\opH^j(G_{s+f},M^{(s)}\otimes H^0(\la)) \neq 0.$$   
Hence, there exists $\ga \in X(T)_+$ such that $L(\ga)$ is a composition factor of 
$\opH^j(G_{s + f},M^{(s)}\otimes H^0(\la))^{(-(s + f))}$ and $\Ext_{G}^i(V(\la),L(\ga)) \neq 0$.  By \cite[Prop. II.6.20]{Jan}, $\la \leq \ga$.  Therefore, 
\begin{equation}\label{weightinequality}
\langle\la,\ta^{\vee}\rangle \leq \langle\ga,\ta^{\vee}\rangle.
\end{equation}

The preceding discussion holds for all primes. For the remainder of the proof, we consider the two cases separately.   For part (a), by (\ref{Gr:weight-even-rough})  
along with (\ref{weightinequality}), we have
\begin{equation}\label{even-one}
2^{s+f}\langle\la,\ta^{\vee}\rangle \leq 2^{s+f}\langle\ga,\ta^{\vee}\rangle
\leq 2^sb(M) + \langle\la,\ta^{\vee}\rangle + m\cdot 2^{s + f}.
\end{equation}
Using our assumption that $s \geq m$  and the assumption on $f$ (which implies that $b(M) \leq 2^{f-1} - 1$), from (\ref{even-one}), we get
\begin{equation}\label{even-two}
(2^{s+f} - 1)\langle\la,\ta^{\vee}\rangle \leq 2^s(2^{f-1} - 1) + s \cdot 2^{s + f}.
\end{equation}
Set $t = \lceil\log_2(\langle\la,\ta^{\vee}\rangle + 1)\rceil$.  Then $2^{t-1} \leq \langle\la,\ta^{\vee}\rangle < 2^t$.  Substituting this into (\ref{even-two}), we get
$$
2^{t-1} \leq \frac{2^{s+f-1} - 2^s}{2^{s+f} - 1} + \frac{2^{s+f}}{2^{s + f} - 1}\cdot s.
$$
By Lemma \ref{L:numerical}, $s \geq t$.  We may now apply Proposition \ref{Gr:weight}.  Note that $f > t(M)$. From (\ref{Gr:weight-even}) and (\ref{weightinequality}) we obtain
$$
2^{t-1} \leq \langle\la,\ta^{\vee}\rangle < m - (s-t).
$$
Using $m \leq s $, we get
$
2^{t-1} < t,
$
which contradicts the fact that $x \leq 2^{x-1} $ for any integer $x$. 

For part (b), we distinguish two cases. \\\\
{\bf Case 1: } $s= m/(p-2)$ and $f \geq \ceil{\log_p (b(M)+1)}+1.$ 
\vskip .25cm 

By (\ref{Gr:weight-odd-rough}) along with (\ref{weightinequality}), we have
\begin{equation}\label{odd-one-=}
p^{s+f}\langle\la,\ta^{\vee}\rangle \leq p^{s+f}\langle\ga,\ta^{\vee}\rangle
\leq p^sb(M) + \langle\la,\ta^{\vee}\rangle + m\cdot p^{s + f}.
\end{equation}
Using our assumption that $s = \frac{m}{p-2}$  and the assumption on $f$ (which implies that $b(M) \leq p^{f-1} - 1$), from (\ref{odd-one-=}), we get
\begin{equation}\label{odd-two-=}
(p^{s+f} - 1)\langle\la,\ta^{\vee}\rangle \leq p^s(p^{f-1} - 1) + s(p-2) \cdot p^{s + f}.
\end{equation}
Set $t = \lceil\log_p(\langle\la,\ta^{\vee}\rangle + 1)\rceil$.  Then $p^{t-1} \leq \langle\la,\ta^{\vee}\rangle < p^t$.  Substituting this into (\ref{odd-two-=}), we get
$$
p^{t-1} \leq \langle\la,\ta^{\vee}\rangle \leq \frac{p^{s+f-1} - p^s}{p^{s+f} - 1} + \frac{p^{s+f}}{p^{s + f} - 1}\cdot s(p-2).
$$
By Lemma \ref{L:numerical}, $s \geq t$.  We may now apply Proposition \ref{Gr:weight}.  Since $f > t(M)$ one obtains from (\ref{Gr:weight-odd}) and (\ref{weightinequality})
$$
p^{t-1} \leq p^{t-1}\la_{t-1} \leq \langle\la,\ta^{\vee}\rangle < m - (s - t + 1)(p-2) + \la_{t-1}
$$
or
$$
p^{t-1} - 1 \leq (p^{t-1} - 1)\la_{t-1} < m - (s - t + 1)(p-2).
$$
Using $m = s(p-2)$, we get
$$
p^{t-1} -1 < s(p-2)  - (s - t + 1)(p-2) = (t-1)(p-2) 
$$
or
$$
p^{t-1} \leq (t-1)(p-2).
$$
This is a contradiction for $t = 1$, and, for $t > 1$, we conclude that $p^{t-1} < p(t-1)$
which contradicts the fact that $p^x \geq px$ for all integers $x$.
\\\\
{\bf Case 2: } $s> m/(p-2)$ and $f \geq \ceil{\log_p (b(M)+1)}.$ 
\vskip .25cm 

Using our assumption that $s> \frac{m}{p-2}$ (or $m \leq s(p-2) -1$) and the assumption on $f$ (which implies that $b(M) \leq p^f - 1$), from (\ref{odd-one-=}), we get
\begin{equation}\label{odd-two}
(p^{s+f} - 1)\langle\la,\ta^{\vee}\rangle \leq p^s(p^f - 1) + (s(p-2) - 1)\cdot p^{s + f}.
\end{equation}
Again set $t = \lceil\log_p(\langle\la,\ta^{\vee}\rangle + 1)\rceil$.    Substituting this into (\ref{odd-two}), we get
$$
p^{t-1} \leq \langle\la,\ta^{\vee}\rangle \leq \frac{p^{s+f} - p^s}{p^{s+f} - 1} + \frac{p^{s+f}}{p^{s + f} - 1}\cdot (s(p-2) - 1).
$$
By Lemma \ref{L:numerical}, $s\geq t$.  We may now apply Proposition \ref{Gr:weight}.  From (\ref{Gr:weight-odd}) and (\ref{weightinequality}) we obtain
$$
p^{t-1} \leq p^{t-1}\la_{t-1} \leq \langle\la,\ta^{\vee}\rangle \leq m - (s - t + 1)(p-2) + \la_{t-1}
$$
or
$$
p^{t-1} - 1 \leq (p^{t-1} - 1)\la_{t-1} \leq m - (s - t + 1)(p-2).
$$
Using $m \leq s(p-2) - 1$, we have
$$
p^{t-1} -1 \leq s(p-2) - 1 - (s - t + 1)(p-2) = (t-1)(p-2) - 1
$$
or
$$
p^{t-1} \leq (t-1)(p-2).
$$
As in Case 1, one obtains a contradiction.
\end{proof}


\section{Vanishing Ranges for $G({\mathbb F}_{q})$-Cohomology}

\subsection{} We apply  Proposition \ref{vanishing} to determine vanishing ranges of the cohomology ring $\operatorname{H}^{\bullet}(G({\mathbb F}_{q}),k)$. Let $M = k$ be the trivial module.  
Then $b(k) = 0.$  It follows that, for all $0 \neq \la \in X(T)_{+}$, 
$\Ext_{G}^m(V(\la)^{(r)},H^0(\la)) = 0$ for all $1 \leq m < r$ when $p = 2$ or $1 \leq m < r(p-2)$ when $p$ is odd.  In addition, when $\la = 0$, we have $\Ext_{G}^m(k,k) = 0$ for all $m > 0$.  

Now by \cite[Cor. 2.6(B)]{BNP8}, we obtain the following result which removes the condition $p > h$ in \cite[Thm. 4.4]{BNP8}.  This also provides a generalization of a theorem of Hiller (cf. \cite[Thm. 4]{H}).

\begin{theorem}\label{T:vanishingrange} Assume $r \geq 1$ and $q=p^{r}$.  
\begin{itemize}
\item[(a)] If $p=2$, then $\opH^m(G({\mathbb F}_{q}),k) = 0$ for $0 < m < r$.
\item[(b)] If $p > 2$, then $\opH^m(G({\mathbb F}_{q}),k) = 0$ for $0 < m < r(p-2)$.
\end{itemize}
\end{theorem}


\subsection{Examples} We illustrate the sharpness of the bounds in Proposition \ref{vanishing} and Theorem \ref{T:vanishingrange} through the following examples.

\begin{example} Let $G = SL_2$ (and recall the weight notation given in Section \ref{SS:notation}).  For any $m$, the dimension of $\opH^m(G({\mathbb F}_{q}),k)$ can be computed from Theorem 2.6 of \cite{Car}. In particular, Theorem \ref{T:vanishingrange} follows from that result. Further, it follows that the bounds in Theorem \ref{T:vanishingrange} are sharp.   For $p = 2$, $\opH^r(\gfpr,k)$ is non-zero, and similarly, for $p > 2$, $\opH^{r(p-2)}(\gfpr,k)$ is non-zero.

Consider now Proposition \ref{vanishing} (still with $G = SL_2$).  Assume that $p = 2$.
Using the five term sequence from the LHS spectral sequence for $G_1 \unlhd G$, one can show that
\begin{align*}
\Ext^1_{G}(V(2)^{(1)},H^0(2)) &\cong \Hom_{G/G_1}(V(2)^{(1)},\opH^1(G_1,H^0(2)))\\
	&\cong \Hom_{G}(V(2),\opH^1(G_1,H^0(2))^{(-1)})\\
	&\cong \Hom_{G}(V(2), \ind_{B}^{G}\opH^1(B_1,2)^{(-1)})\\
	&\cong \Hom_{G}(V(2),H^0(2)) \neq 0.
\end{align*}
This corresponds to the case $s = 0$, $m = 1$, and $f = 1$ in Proposition \ref{vanishing} and shows that we must take $s$ greater than or equal to $m$ in general.

More generally, suppose that $s \geq 1$.  For any dominant weight $\la$, since $\opH^{\bullet}(G_{s+f},H^0(\la))^{(-1)}$ admits a good filtration, the LHS spectral sequence for $G_{s+f} \unlhd G$ collapses to give:
$$
E_2^{0,m} = \Hom_{G/G_{s+f}}(V(\la)^{(s)},\opH^m(G_{s+f},H^0(\lambda))) \cong 
\Ext_{G}^m(V(\la)^{(s+f)},H^0(\la)).
$$
However, by \cite[Theorem 4.3.1]{N}, $\opH^m(G_{s+f},H^0(\la))^{(-(s+f))} \cong \ind_{B}^{G} \opH^m(B_{s+f},\la)^{(-(s+f))}$, and  
\begin{eqnarray*}
\Ext_{G}^m(V(\la)^{(s+f)},H^0(\la))&\cong &\Hom_{G/G_{s+f}}\left(V(\la)^{(s+f)}, \ind_{B}^{G}\left(\opH^m(B_{s+f},\la)^{(-(s+f))}\right)^{(s+f)}\right)  \\
&\cong& \Hom_{G}\left(V(\la), \ind_{B}^{G} \opH^m(B_{s+f},\la)^{(-(s+f))}\right)\\
&\cong & \Hom_{B}\left(V(\la), \opH^{m}(B_{s+f},\la)^{(-(s+f))}\right).
\end{eqnarray*} 
Taking in particular,  $\la = 2^{s+f}\omega$, one sees using \cite[Theorem 4.2.3]{N} that $E_{2}^{m,0}\neq 0$ if and only if 
$2^{s+f}\omega$ is a weight of  $\opH^{m}(B_{s+f},k)^{(-(s+f))} \otimes \omega$, or equivalently, 
$(2^{s+f}-1)2^{s+f}\omega$ is a weight of $\opH^{m}(B_{s+f},k)$. 

A weight of $\opH^{m}(B_{s+f},k)$ is weight of $\opH^{m}(U_{s+f},k)$ which is $T_{s+f}$ equivariant. We want be able to express 
$(2^{s+f}-1)2^{s+f}=2a_{i_{1}}2^{i_{1}}+2a_{i_{2}}2^{i_{2}}+ \dots +2a_{i_{u}}2^{i_{u}}$ (or equivalently 
$(2^{s+f}-1)2^{s+f-1}=a_{i_{1}}2^{i_{1}}+a_{i_{2}}2^{i_{2}}+\dots + a_{i_{u}}2^{i_{u}}$) where $m=a_{i_{1}}+a_{i_{2}}+\dots + a_{i_{u}}$. 
Now one can take the $2$-adic expansion of $(2^{s+f}-1)2^{s+f-1}$ and see that it contains precisely $s+f$ nonzero places. 
Therefore, $\Ext_{G}^m(V(\la)^{(s+f)},H^0(\la))\neq 0$ when $m=s+f$, and again we see the sharpness of the condition.

Now assume that $p = 3$.  Since the module $H^0(1)^{(1)} \cong H^0(3)$ has a filtration with two simple composition factors (from top to bottom): $L(1)$ and $L(3)$,
$\Ext_{G}^1(V(1)^{(1)},H^0(1)) \neq 0$.  In the context of Proposition \ref{vanishing}, we have $m = 1$ and $s + f = 1$.  Again, we see that the conditions on $s$ and $s + f$ are sharp. Note further that $\opH^1(\gfp,k) \neq 0$ (as observed in the first part of this example), but $\opH^1(G,k) = 0$.
\end{example} 

\begin{example} Suppose that $\Phi$ has type $C_n$ and $p > 2n$.   It was shown in \cite[Lem. 5.1]{BNP8} that $\Ext_{G}^{s(p-2)}(V(\la)^{(s)},H^0(\la)) \neq 0$ for $\la = (p-2n)\omega_1$. In the context of Proposition \ref{vanishing}, this corresponds to 
$s = m/(p-2)$ and $s + f = s$ and shows the sharpness of the conditions.  It was further shown in \cite[Thm. 5.2]{BNP8} that $\opH^{r(p-2)}(\gfpr,k) \neq 0$ which demonstrates the sharpness of the bounds in Theorem \ref{T:vanishingrange} (for $p > 2n$).
\end{example}


\section{Generic Cohomology}\label{S:generic}

\subsection{Generic Cohomology: $m$ general}\label{SS:generic} In the following theorem, we prove that the restriction map $\text{res}:\text{H}^{m}(G,M^{(s)})\rightarrow \text{H}^{m}(G({\mathbb F}_{q}),M^{(s)})$ yields 
an isomorphism for sufficiently large $s$ and $r$. Effective lower bounds are provided in this result. 

\begin{theorem}\label{theorem:generic3}  Assume that $m \geq 0$ and let $q=p^{r}$. Let $M$ be a rational $G$-module, and consider the 
restriction map in cohomology: $\operatorname{res}:\operatorname{H}^{m}(G,M^{(s)})\rightarrow \operatorname{H}^{m}(G({\mathbb F}_{q}),M^{(s)})$.
\begin{itemize}
\item[(a)] If $p=2$, $e = m$,  and  $f =  \lceil\log_2(b(M)+1)\rceil$, then $\operatorname{res}$ is an isomorphism and 
$$\opH^m(G({\mathbb F}_{q}), M) \cong \opH^m(G({\mathbb F}_{q}), M^{(s)}) \cong \opH^m(G, M^{(s)}) \text{ for all } s\geq e \text{ and } r \geq e+f+1.$$
\item[(b)] If $p$ is odd, $e = m/(p-2)$, and  $f =  \lceil\log_p (b(M)+1)\rceil$, then $\operatorname{res}$ is an isomorphism and 
$$\opH^m(G({\mathbb F}_{q}), M) \cong \opH^m(G({\mathbb F}_{q}), M^{(s)}) \cong \opH^m(G, M^{(s)}) \text{ for all } s\geq e \text{ and } r \geq \lfloor e \rfloor +f+1.$$
\end{itemize}
\end{theorem}

\begin{proof}The first isomorphism follows from the fact that the Frobenius morphism is an automorphism on the finite group. We need to verify the second isomorphism which 
involves the restriction map. Define $N$ via the exact sequence 
$0 \rightarrow k \rightarrow \text{ind}_{G({\mathbb F}_{q})}^{G} k \rightarrow N \to 0$. By tensoring with $M^{(s)}$ and using the tensor identity, we have a short exact sequence 
$$0\rightarrow M^{(s)} \rightarrow \text{ind}_{G({\mathbb F}_{q})}^{G} M^{(s)} \rightarrow N\otimes M^{(s)} \rightarrow 0.$$
The long exact sequence of cohomology (cf. (\ref{eq:lesrestriction})) yields an exact sequence
$$\cdots \rightarrow  \opH^{m-1}(G, N\otimes M^{(s)}) \to \opH^m(G, M^{(s)}) \stackrel{\res}{\longrightarrow}  \opH^m(G({\mathbb F}_{q}), M^{(s)}) \rightarrow \opH^{m}(G, N\otimes M^{(s)})\rightarrow\cdots .$$
It follows from \cite[Proposition 2.4.1]{BNP8}  that $N$ has a filtration with sections of the form $H^0(\mu) \otimes H^0(\mu^*)^{(r)}$ with $\mu \in X(T)_+$ and $\mu \neq 0$. It suffices therefore to show that 
\begin{equation*}\label{vanish}
\opH^m(G, M^{(s)}\otimes H^0(\mu) \otimes H^0(\mu^*)^{(r)} ) \cong
\Ext^m_G(V(\mu)^{(r)}, M^{(s)}\otimes H^0(\mu))   
\end{equation*}
vanishes, whenever $\mu \neq 0$, $s\geq e$, and $r \geq e+f+1$ when $p=2$, $r \geq \lfloor e \rfloor +f+1$ when $p$  is odd. 
This fact follows immediately from Proposition~\ref{vanishing}.
\end{proof}

\begin{remark}\label{R:genericproof}
We can give a new proof of Theorem \ref{theorem:generic2} without the assumption that $\operatorname{H}^{n}(G_{1},k)^{(-1)}$ has a good filtration for $n \leq m$. 
Assume that $s \geq m$ when $p=2$ and $s \geq m/(p-2)$ when $p$ is odd. In addition, assume that $r >> s+ b(M)$.  Theorem \ref{theorem:generic3} now gives rise 
to the following isomorphisms.
$$\opH^m(G, M^{(s)}) \cong \opH^m(G({\mathbb F}_{q}), M^{(s)}) \cong \opH^m(G({\mathbb F}_{q}), M) \cong \opH^m(G({\mathbb F}_{q}), M^{(s+1)}) \cong \opH^m(G, M^{(s+1)}).$$
\end{remark}


\subsection{Generic Cohomology: $m=1$, $p$ odd} 

In the case when $p$ is odd and $m=1$ of Theorem~\ref{theorem:generic3}, one has $e=m/(p-2)>0$.  The next result shows that in this case 
we can lower the bound to $e=0$.

\begin{theorem}\label{theorem:generic4}  Assume that $m=1$ and let $q=p^{r}$ with $p$-odd. Let $M$ be a rational $G$-module, and consider the restriction map in 
cohomology: $\operatorname{res}:\operatorname{H}^{m}(G,M^{(s)})\rightarrow \operatorname{H}^{m}(G({\mathbb F}_{q}),M^{(s)})$. If $e = 0$, and  $f = \lceil\log_p (b(M)+1)\rceil$, 
then $\operatorname{res}$ is an isomorphism and 
$$\opH^m(G({\mathbb F}_{q}), M) \cong \opH^m(G({\mathbb F}_{q}), M^{(s)}) \cong \opH^m(G, M^{(s)}) \text{ for all } s\geq e \text{ and } r \geq f+1$$
when any of the following conditions hold: 
\begin{itemize} 
\item[(a)] $\Phi$ is not of type $A_{1}$; 
\item[(b)] $\Phi$ is of type $A_{1}$ and $p\geq 5$; 
\item[(c)] $\Phi$ is of type $A_{1}$, $p=3$, and $r\geq 2$.  
\end{itemize} 
\end{theorem}

\begin{proof} We first begin by applying the long exact sequence of cohomology (cf. (\ref{eq:lesrestriction})) as in Theorem~\ref{theorem:generic3}. In order to prove the 
assertion, it suffices to show that, for $M$ an arbitrary finite-dimensional rational $G$-module,
\begin{equation}\label{eq:hom} 
\text{Hom}_{G}(V(\mu)^{ ( r )},M\otimes H^{0}(\mu))=0
\end{equation} 
\begin{equation}\label{eq:ext1} 
\text{Ext}^{1}_{G}(V(\mu)^{ ( r )},M\otimes H^{0}(\mu))=0
\end{equation} 
with $\mu\neq 0$. Without a loss of generality we may assume that $M=L(\sigma)$ where $\sigma\in X(T)_{+}$. Under our assumptions, 
$r \geq  \lceil\log_p (b(L(\sigma))+1)\rceil+1$ which implies that $p^{r}\geq p\cdot (\langle \sigma, \tilde{\alpha}^{\vee} \rangle+1)$ or $p^{r-1}-1\geq \langle \sigma, \tilde{\alpha}^{\vee} \rangle$. 
Therefore, $\sigma\in X_{r-1}(T)$. 

Now consider the LHS spectral sequence: 
\begin{equation*} 
E_{2}^{i,j}=\text{Ext}^{i}_{G/G_{r}}(V(\mu)^{( r)}, \text{Ext}^{j}_{G_{r}}(L(\sigma),H^{0}(\mu)))\Rightarrow 
\text{Ext}^{i+j}_{G}(V(\mu)^{( r)},L(\sigma)\otimes H^{0}(\mu)). 
\end{equation*} 
First observe that for $j=0,1$, $\text{Ext}^{j}_{G_{r}}(L(\sigma),H^{0}(\mu))\cong \text{ind}_{B/B_{r}}^{G/G_{r}}\ \text{Ext}^{j}(L(\sigma),\mu)$. 
Write $\mu=\mu_{0}+p^{r}\mu_{1}$ where $\mu_{0}\in X_{r}(T)$ and $\mu_{1}\in X(T)_{+}$. Then,  
for $j=0$,
\begin{equation*}
\text{Hom}_{G_{r}}(L(\sigma),H^{0}(\mu))\cong 
\begin{cases} 
\text{ind}_{B/B_{r}}^{G/G_{r}} p^{r}\mu_{1} & \text{$\sigma=\mu_{0}$} \\
0  & \text{$\sigma\neq \mu_{0}$}.  
\end{cases} 
\end{equation*}  

Therefore, $\text{Hom}_{G}(V(\mu)^{ ( r )},L(\sigma) \otimes H^{0}(\mu))\cong 
\text{Hom}_{G}(V(\mu),\text{ind}_{B}^{G} \mu_{1})$.  This is non-zero if and only if $\mu=\mu_{1}$ (i.e., $\mu_{0}+(p^{r}-1)\mu_{1}=0$) or 
equivalently when $\mu=0$.  Consequently, (\ref{eq:hom}) holds. 

The argument above shows that $\text{Hom}_{G_{r}}(L(\sigma),H^{0}(\mu))$ has a good filtration, thus 
$E_{2}^{i,0}=0$ for $i>0$, and so $\text{Ext}^{1}_{G}(V(\mu)^{( r)},L(\sigma)\otimes H^{0}(\mu))\cong E_{2}^{0,1}$. Hence, we are reduced to showing that 
\begin{equation*} 
\text{Hom}_{G/G_{r}}(V(\mu)^{( r)}, \text{Ext}^{1}_{G_{r}}(L(\sigma),H^{0}(\mu)))=0.
\end{equation*} 
From our analysis, we have 
\begin{equation*}
\text{Hom}_{G/G_{r}}(V(\mu)^{( r)}, \text{Ext}^{1}_{G_{r}}(L(\sigma),H^{0}(\mu)))\cong 
\text{Hom}_{G}(V(\mu),\text{ind}_{B}^{G}(\text{Ext}^{1}_{B_{r}}(L(\sigma),\mu_{0})^{(-r)}\otimes \mu_{1})).
\end{equation*}

Next we need to analyze $\text{Ext}^{1}_{B_{r}}(L(\sigma),\mu_{0})$. By using the short exact sequence 
$0\rightarrow \mu_{0}\rightarrow I_{r}(\mu_{0}) \rightarrow Q \rightarrow 0$ where $I_{r}(\mu_{0})$ is the 
injective hull of $\mu_{0}$ in the category of $B_{r}$-modules, and applying the long exact sequence in cohomology, one sees that 
$\text{Ext}^{1}_{B_{r}}(L(\sigma),\mu_{0})\cong \text{Hom}_{B_{r}}(L(\sigma),Q)$. 
The $T_{r}$ weights in the socle of $Q$ are given by $\mu_{0}+p^{l}\alpha$ where $l=0,1,\dots, r-1$ and $\alpha\in \Pi$ (cf. \cite[Section 2]{BNP4}). 
Consequently, if $\text{Hom}_{B_{r}}(L(\sigma),Q)\neq 0$, then there exists a weight $\delta$ of $L(\sigma)$ such that 
$$\mu_{0}+p^{l}\alpha-\delta=p^{r}\theta$$ 
for some weight $\theta$. Furthermore, we can conclude that if 
$\text{Hom}_{G}(V(\mu), \text{ind}_{B}^{G}( \text{Ext}^{1}_{B_{r}}(L(\sigma),\mu_{0})^{(-r)}\otimes \mu_{1}))\neq 0$ 
then $\mu=\theta+\mu_{1}$ where $\mu_{0}+p^{l}\alpha-\delta=p^{r}\theta$ for some $l=0,1,\dots, r-1$ and $\alpha\in \Pi$. 
From these two conditions, one can deduce that 
$$(p^{r}-1)\mu=p^{l}\alpha-\delta.$$
Since $0 \neq \mu\in X(T)_{+}$, we have 
\begin{equation}\label{eq:muinequality}
(p^{r}-1) \leq (p^{r}-1)\langle \mu,\tilde{\alpha}^{\vee} \rangle. 
\end{equation} 
Furthermore, 
\begin{equation}\label{eq:muinequality2}
(p^{r}-1)\langle \mu,\tilde{\alpha}^{\vee} \rangle \leq  p^{l}\langle \alpha,\tilde{\alpha}^{\vee} \rangle + \langle -\delta,\tilde{\alpha}^{\vee} \rangle 
 \leq  p^{l}\langle \alpha,\tilde{\alpha}^{\vee} \rangle + (p^{r-1}-1) 
\leq  p^{r-1}(\langle \alpha, \tilde{\alpha}^{\vee} \rangle +1)-1.
\end{equation} 
Combining (\ref{eq:muinequality}) and (\ref{eq:muinequality2}), it follows that 
\begin{equation} \label{eq:maininequality}
p^{r}\leq p^{r-1}(\langle \alpha, \tilde{\alpha}^{\vee} \rangle +1).
\end{equation}

By assumption, $p\geq 3$, and $\langle \alpha, \tilde{\alpha}^{\vee} \rangle \leq 1$ unless $\Phi$ is of type $A_{1}$ (in which case $\langle \alpha, \tilde{\alpha}^{\vee} \rangle=2$, since $\al = \ta$ is the unique simple root). 
One can conclude that we reach a contradiction in all cases except when $\Phi$ is of type $A_{1}$ and $p=3$. This verifies the theorem under conditions (a) and (b). 

In the exceptional case of type $A_{1}$ with $p = 3$, one can garner more information by using the equations above.   The exception occurs only when equality holds in (\ref{eq:maininequality}).   However, for that to occur, one must have $\mu= 1$  and $\sigma= 3^{r-1}-1$. Observe that 
$L(\sigma) \cong L(2)^{(r-1)}\otimes L(2)^{(r-2)} \otimes \dots \otimes L(2)$ for $r\geq 2$. One can show inductively using the LHS spectral sequence (with $B_{r-1}\unlhd B_{r}$) that 
$\text{Ext}^{1}_{B_{r}}(L(\sigma),\mu)=0$ for $r\geq 2$. This verifies (\ref{eq:ext1}), and thus the theorem also holds under condition (c). 
 \end{proof} 

\begin{remark} An alternative approach to this theorem can be provided for $p>3$ by using results in Section~\ref{rat}. By Theorem~\ref{theorem:generic3} 
we have that $\operatorname{res} : \opH^1(G,M^{(s)}) \to \opH^1(\gfpr,M^{(s)})$ is an isomorphism for $s\geq 1$ and $r\geq \lfloor 1/(p-2) \rfloor +f+1 = f + 1$. But by 
Corollary~\ref{corollary:generic3}, one has $\text{H}^{1}(G,M^{(1)})\cong \text{H}^{1}(G,M)$. 
\end{remark} 

\begin{remark} 
We remark that in the case when $\Phi$ is of type $A_{1}$, $p=3$, $r=1$ and $M=k$ the restriction map 
$\operatorname{res}:\operatorname{H}^{1}(G,k)\rightarrow \operatorname{H}^{1}(G({\mathbb F}_{p}),k)$
is not an isomorphism because $\operatorname{H}^{1}(G,k)=0$ while $\operatorname{H}^{1}(G({\mathbb F}_{p}),k)\cong k$. 
This is consistent with the fact that $\text{Ext}^{1}_{G_{1}}(L(0),L(1))\cong L(1)^{(1)}$ which shows that 
$\text{Hom}_{G/G_{1}}(V(1)^{( 1)}, \text{Ext}^{1}_{G_{1}}(L(0),H^{0}(1)))\cong k$.  
\end{remark}


\subsection{Generic Cohomology: Type $A_{1}$, $p$ odd}

The next result demonstrates that in the case when $\Phi$ is of type $A_{1}$ we can lower the bound to $e= \lceil(m-1)/(p-2)\rceil$. 

\begin{theorem}\label{theorem:generic5}  Assume that $\Phi$ is of type $A_{1}$ with $p$-odd. Let $M$ be a rational $G$-module, and consider the restriction map in 
cohomology: $\operatorname{res}:\operatorname{H}^{m}(G,M^{(s)})\rightarrow \operatorname{H}^{m}(G({\mathbb F}_{q}),M^{(s)})$. If $e = \lceil (m-1)/(p-2) \rceil$, and  $f = \lceil\log_p (b(M)+1)\rceil$, 
then $\operatorname{res}$ is an isomorphism and 
$$\opH^m(G({\mathbb F}_{q}), M) \cong \opH^m(G({\mathbb F}_{q}), M^{(s)}) \cong \opH^m(G, M^{(s)}).$$
when any of the following conditions holds: 
\begin{itemize} 
\item[(a)] $p\geq 5$, $s\geq e$ and $r\geq e + f+1$;  
\item[(b)] $p=3$,  $s\geq m-1$ and $r\geq m+1+ \lfloor \log_3 (b(M)+1) \rfloor$.    
\end{itemize} 
\end{theorem}

\begin{proof} In the type $A_1$-case, per Section \ref{SS:notation}, a weight $\si \in X(T)$ will be considered as an integer.
As in Theorem~\ref{theorem:generic4}, we are reduced to considering 
$$E_{m}=\text{Ext}^{m}_{G}(V(\mu)^{(r)},L(\sigma)^{(e)}\otimes H^{0}(\mu))$$ for $\mu, \si \in X(T)_+$ with $\mu\neq 0$. First, consider the case when $p\geq 5$.  We want to show that $E_{m}$ is zero when 
$r\geq  e  +\lceil\log_p (\sigma+1)\rceil+1$. 
From the LHS spectral sequence we have 
\begin{equation} \label{ss1}
E_{2}^{i,j}=\text{Ext}^{i}_{G/G_{r}}(V(\mu)^{(r)}, \text{Ext}^{j}_{G_{r}}(L(\sigma)^{(e)},H^{0}(\mu)))\Rightarrow 
\text{Ext}^{i+j}_{G}(V(\mu)^{(r)},L(\sigma)^{(e)}\otimes H^{0}(\mu)). 
\end{equation} 

Consider the $G$-module $\text{Ext}_{G_{r}}^{\bullet}(L(\sigma)^{(e)},H^{0}(\mu))$. There exists a spectral sequence \cite[II.12.2]{Jan}
\begin{equation} \label{ss2}
\widetilde{E}_{2}^{i,j}=R^{i}\text{ind}_{B^{(r)}}^{G^{(r)}}\ \text{Ext}^{j}_{B_{r}}(L(\sigma)^{(e)},\mu)\Rightarrow 
\text{Ext}^{i+j}_{G_{r}}(L(\sigma)^{(e)},H^{0}(\mu)). 
\end{equation} 
Next we need to analyze the $T$-weights of $\text{Ext}^{\bullet}_{B_{r}}(L(\sigma)^{(e)},\mu)$. These weights 
are $T_{r}$-invariant weights of the form $p^{r}\delta=p^{e}\gamma+\mu+\Gamma$ where 
$\gamma$ is a weight of $L(\sigma)^{*}$ and $\Gamma$ is a weight of $\text{H}^{m}(U_{r},k)$. Since we are in type $A_{1}$, $\mu+\Gamma\in X(T)_{+}$. Furthermore, $\gamma$ is a weight of $L(\sigma)$ so 
$-p^{f}<\gamma < p^{f}$ and 
$$-p^{r}<-p^{e+f} <p^{e}\gamma < p^{e+f} <p^{r}.$$ 
Hence, $- p^r < p^r\delta$ and so $\delta > 0$.  In other words,
$\delta$ and $p^{r}\delta$ must be in $X(T)_{+}$. 

It follows that $\widetilde{E}_{2}^{i,j}=0$ for $j>0$ and the spectral sequence (\ref{ss2}) collapses and yields, for $n\geq 0$, 
\begin{equation} 
\text{Ext}^{n}_{G_{r}}(L(\sigma)^{(e)},H^{0}(\mu))\cong \text{ind}_{B^{(r)}}^{G^{(r)}}\text{Ext}^{n}_{B_{r}}(L(\sigma)^{(e)},\mu). 
\end{equation} 
Moreover, $\text{Ext}^{\bullet}_{G_{r}}(L(\sigma)^{(e)},H^{0}(\mu))$ has a good filtration. The spectral sequence (\ref{ss1}) collapses 
and 
\begin{equation} 
E_{m}\cong \text{Hom}_{G/G_{r}}(V(\mu)^{(r)}, \text{Ext}^{m}_{G_{r}}(L(\sigma)^{(e)},H^{0}(\mu)))\cong  \text{Hom}_{B/B_{r}}(V(\mu)^{(r)}, \text{Ext}^{m}_{B_{r}}(L(\sigma)^{(e)},\mu)). 
\end{equation} 

Recall from Sections \ref{SS:spectral} and \ref{SS:weight} that a typical weight of $\text{Ext}^{m}_{B_{r}}(L(\sigma)^{(e)},\mu)$ is of the form 
$$p^{e}\gamma+\mu+2(a_{0}+b_{0})+2(a_{1}+b_{1})p+\dots+2(a_{r}+b_{r})p^{r}$$ 
where $m=\sum_{k=0}^{r}(2a_{k}+b_{k})$ for nonnegative integers $a_i$, $b_i$ with $a_{0}=0=b_{r}$. Note also that, since we are in type $A_1$, each $b_i \leq 1$, and so $b_i \in \{0, 1\}$.  If $E_{m}\neq 0$, then 
$$p^{r}\mu=p^{e}\gamma+\mu+2(a_{0}+b_{0})+2(a_{1}+b_{1})p+\dots+2(a_{r}+b_{r})p^{r}$$
or equivalently 
\begin{equation}\label{Eq:muexpansion}
(p^{r}-1)\mu=p^{e}\gamma+2(a_{0}+b_{0})+2(a_{1}+b_{1})p+\dots+2(a_{r}+b_{r})p^{r}.
\end{equation}
Write $(p^{r}-1)\mu=\sum_{i\geq 0}d_{i}p^{i}$ in its $p$-adic expansion. By using induction on $r$, we have 
\begin{equation}\label{Eq:sumlowerbound}
r(p-1)\leq  \sum_{i\geq 0}d_{i}.
\end{equation}
Moreover, using (\ref{Eq:muexpansion}), $m = \sum_{k=0}^{r}(2a_{k}+b_{k})$, and $\ga < p^f$, one can show that  
\begin{equation}\label{Eq:sumupperbound}
\sum_{i\geq 0} d_{i}\leq m+r+(p-2)f.
\end{equation}
Combining (\ref{Eq:sumlowerbound}) and (\ref{Eq:sumupperbound}) yields 
\begin{equation}\label{Eq:r-fbound}
(r-f)(p-2)\leq m. 
\end{equation} 
On the other hand, by our hypothesis,
$$r\geq e+f+1 \geq (m-1)/(p-2) +f +1$$
or $(r-f-1)(p-2)\geq (m-1)$. Combining this with (\ref{Eq:r-fbound}) yields $(p-2) \leq 1$ which is a contradiction. 

Now assume that $p=3$. We are assuming that $r\geq m+1+ \lfloor \log_3 (b(M)+1) \rfloor$ or $r - (1 + \lfloor \log_3 (b(M)+1) \rfloor) \geq m$.  Since $f = \lceil \log_3(b(M) + 1)\rceil \leq 1 + \lfloor \log_3 (b(M)+1)\rfloor$, we have $r - f \geq m$.

We can use the prior analysis to say that, if $E_{m}\neq 0$, then 
\begin{equation}\label{Eq:muexpansion2}(p^{r}-1)\mu=p^{m-1}\gamma+2(a_{0}+b_{0})+2(a_{1}+b_{1})p+\dots+2(a_{r}+b_{r})p^{r}.
\end{equation}
Further, knowing that $r \geq m$, (\ref{Eq:r-fbound}) still holds.  With $p = 3$, this becomes $r - f \leq m$.  Combining that with the condition $r - f \geq m$ forces $r - f = m$. 

When $r-f=m$, we must have equality in (\ref{Eq:r-fbound}) and hence in (\ref{Eq:sumupperbound}).  However, equality in (\ref{Eq:sumupperbound}) requires that $\sum_{i = 0}^{r}b_i = m$.  Combining that with $m=\sum_{i=0}^{r}(2a_{i}+b_{i})$ implies that $a_{i}=0$ for all $i$. Therefore, (\ref{Eq:muexpansion2}) becomes
$$(p^{r}-1)\mu=p^{r-f-1}\gamma+2(b_{0}+b_{1}p+\dots+b_{r-1}p^{r-1}).$$
As $\mu \neq 0$, the left hand side of this equation is greater than $p^{r}-1$. Since $0\leq b_{i} \leq 1$ for all $i$, it follows that the right hand side is 
less than $2(p^{r}-1)$. Consequently, $\mu=1$. 

Now the $p$-adic expansion of $(p^{r}-1)-2(b_{0}+b_{1}p+\dots+b_{r-1}p^{r-1})$ has $f=r-m$ slots with $(p-1)$, thus 
$\gamma=(p-1)+(p-1)p+\dots +(p-1)p^{f-1} = p^f - 1$.  Hence, $b(M) = p^f - 1$, and we deduce that $\log_3(b(M) + 1) = f$ (an integer).   Therefore,
$\lfloor \log_3 (b(M)+1) \rfloor = f$ and our hypothesis becomes $r\geq m+1+f$ or $r-f\geq m+1$. This contradicts the fact that $r-f=m$. 
\end{proof} 

\subsection{Comparison with the bounds in \cite{CPSvdK}}

In this section, we compare the bounds in Theorems \ref{theorem:generic3}, \ref{theorem:generic4}, and \ref{theorem:generic5} with the Main Theorem of \cite{CPSvdK}.   To do so, we introduce some further notation. Set\\
\\
$c$:  the maximum $m_i$, if $\ta = \sum_i m_i\al_i$, where the sum is over all simple roots $\al_i$.\\
$t$: the exponent of the finite abelian group $X(T)/ \mathbb{Z}\Phi.$\\
\\
The values of $c$ and $t$ for each root system are given in the following table. \\

 \begin{center}
\begin{tabular}{|l|r|r|r|}
$\Phi$&$c$&$t$&$c\cdot t$\\
\hline
$A_n$&$1$&$n+1$&$n+1$\\
$B_n$&$2$&$2$&$4$\\
$C_n$&$2$&$2$&$4$\\
$D_n$&$2$&$2$&$4$\\
$E_6$&$3$&$3$&$9$\\
$E_7$&$4$&$2$&$8$\\
$E_8$&$6$&$1$&$6$\\
$F_4$&$4$&$1$&$4$\\
$G_2$&$3$&$1$&$3$\\
\end{tabular}
\end{center}
\ \\
\ \\
For any dominant weight $\la \in X(T)_+$, set\\\\
$c(\la): $ the maximum $m_i,$ if $\la = \sum_i m_i \alpha_i.$ Note that our $c(\la)$ might be a rational number.\\
$t_p(\la): $ the $p$-part of the order of the image of $\la$ in $X(T)/\mathbb{Z}\Phi.$\\
$d(\la): $ the inner product $\langle \la , \ta^{\vee} \rangle.$\\\\
Note that 
$$d(\la) = 
\begin{cases} m_1 &\mbox{ if $\Phi$ of type }  C_n, E_7, F_4,\\
m_2 &\mbox{ if $\Phi$ of type } B_n, D_n, E_6, G_2\\
m_1 + m_n &\mbox{ if $\Phi$ of type } A_n, n\geq 2.\\
2m_1&\mbox{ if $\Phi$ of type } A_1.
\end{cases}
$$
Note
that $d(\la) \in \mathbb{Z}.$ 
In addition $d(\la) \leq c(\la)$ unless $\Phi $ is of type $A_n.$ In the latter cases $d(\la) \leq  2\cdot c(\la).$\\\\
For any finite dimensional $G$-module $M$, set\\\\
$c(M): $ the maximum of $c(\la)$ for all weights $\la$ of $M.$ \\
$d(M): $ the maximum of $d(\la)$ for all weights $\la$ of $M.$ \\
\\

To facilitate a comparison with \cite{CPSvdK}, we continue to follow the same notation as above.  Both works determine numbers $e, f$ so that if $s \geq e$ and $r \geq \lfloor e\rfloor + f + 1$, then
$$
\opH^m(G({\mathbb F}_{q}),M) \cong \opH^m(G,M^{(s)}).
$$
The tables below give the values of $e$ and $f$ from the two sources.  Note that while $f$ is an integer in both cases, we allow $e$ to be a rational number.  More significantly, we note that the value of ``$f$'' 
appearing in \cite{CPSvdK} is actually $f + 1$ for the $f$ given below.  This difference is accounted for in our condition $r \geq \lfloor e\rfloor + f + 1$, whereas it is stated as $r \geq e + f$ in \cite{CPSvdK}.
\\
\newpage
\begin{center}
 {\bf Comparison for $p=2$}
 \ \\
 \ \\
\begin{tabular}{|l|r|r|}
\hline
&\cite[(6.6) Theorem]{CPSvdK}&Theorem \ref{theorem:generic3}(a) (BNP)\\
&                                               &                                    \\
\hline
&&\\
$e$&$\max \{ctm-1, 0\}$&$m$\\
&&\\
\hline
&&\\
$f$&$\lfloor \log_2( t \cdot c(M)+1)\rfloor +1$&$\lceil \log_2(d(M)+1)\rceil $\\
&&\\
\hline
\end{tabular}
\end{center}
\ \\
\ \\
 \begin{center}
 {\bf Comparison for $p$ odd}
 \ \\
 \ \\
\begin{tabular}{|l|r|r|}
\hline
&\cite[(6.6) Theorem]{CPSvdK}&Theorem \ref{theorem:generic3}(b) (BNP)\\
&                                                &                                   \\
\hline
&&\\
$e$&$\max \{\lfloor \frac{ctm-1}{p-1} \rfloor , \lfloor \frac{ct_p(\la)(m-1)-1}{p-1} \rfloor+1\} $&$\frac{m}{p-2}$\\
&&\\
\hline
&&\\
$f$&$\lfloor \log_p( t \cdot c(M)+1)\rfloor +1$&$\lceil \log_p(d(M)+1)\rceil $\\
&&\\
\hline
\end{tabular}
\end{center}
\ \\
\ \\
{\bf Conclusion: } The bound for $f$ in Theorem \ref{theorem:generic3} is always less than or equal to the bound given in \cite{CPSvdK}.
Similarly, the bound for $e$ given in Theorem \ref{theorem:generic3} is  less than or equal to the bound in \cite{CPSvdK},
with  some exception in the odd prime case when $m=1$ or $\Phi$ is of type $A_1$. In these cases Theorem \ref{theorem:generic4} and Theorem \ref{theorem:generic5} provide new and improved bounds.



\begin{remark}
Note that all the theorems in Section \ref{S:generic} do not involve the constants $c$ and $t$ which depend on the root system. 

It follows that for  type $A_n$ and $M$ not a sum of trivial modules,  the difference between the bound for $f$ in \cite{CPSvdK} minus  the  bound given in Theorem \ref{theorem:generic3} is at least $\lfloor\log_p{((n+1)/2)}\rfloor.$
For $p=2$ and type $A_n$, the bound for $e$ in \cite{CPSvdK} is at least  $n$ times the bound given in Theorem \ref{theorem:generic3}. 
Similarly, for p odd, type $A_n$, $n\geq 2$, and $m \geq p-1$, the bound for $e$ in \cite{CPSvdK} is at least  $n/2$ times the bound given in Theorem \ref{theorem:generic3}.

\end{remark}


\providecommand{\bysame}{\leavevmode\hbox
to3em{\hrulefill}\thinspace}

\end{document}